\documentclass{ws-ijnt}

\usepackage{amsfonts,amsmath,synttree,amssymb}
\usepackage{MnSymbol}
\usepackage{relsize}
\usepackage[mathscr]{eucal}
\usepackage{url}
\usepackage{enumerate}
\usepackage{graphicx}

\allowdisplaybreaks[1]

\PassOptionsToPackage{usenames,dvipsnames,svgnames}{xcolor}  
\usepackage{tikz}
\usetikzlibrary{arrows,positioning,automata}

\DeclareMathOperator{\lcm}{lcm}

\newtheorem{thrm}{Theorem}[section]

\newtheorem{lem}[thrm]{Lemma}
\newtheorem{cor}[thrm]{Corollary}

\newcommand{\labeq}[1]{\label{eq:#1}}
\newcommand{\refeq}[1]{(\ref{eq:#1})}
\newcommand{\labt}[1]{\label{thm:#1}}
\newcommand{\reft}[1]{Theorem~\ref{thm:#1}}
\newcommand{\labl}[1]{\label{lemma:#1}}
\newcommand{\refl}[1]{Lemma~\ref{lemma:#1}}
\newcommand{\labd}[1]{\label{definition:#1}}
\newcommand{\refd}[1]{definition~\ref{definition:#1}}
\newcommand{\labc}[1]{\label{coro:#1}}
\newcommand{\refc}[1]{Corollary~\ref{coro:#1}}
\newcommand{\labj}[1]{\label{conj:#1}}
\newcommand{\refj}[1]{Conjecture~\ref{conj:#1}}

\newcommand{\labs}[1]{\label{sec:#1}}
\newcommand{\refs}[1]{Section~\ref{sec:#1}}
\newcommand{\labf}[1]{\label{fig:#1}}
\newcommand{\reff}[1]{Figure~\ref{fig:#1}}

\makeatletter

\newcommand{\Rmnum}[1]{\expandafter\@slowromancap\romannumeral #1@}
\makeatother

\newcommand{\dimh}[1]{\hbox{dim$_{\hbox{H}}$}\left( #1\right)}

\newcommand{\q}[1]{q_1 \cdots q_{ #1 }}

\newcommand{\NN}{\mathbb{N}_2^{\mathbb{N}}}

\newcommand{\measNNz}{\mathscr{M} \left( \mathbb{N}_0^{\mathbb{N}} \right)}

\newcommand{\ppq}{\psi_{P,Q}}

\newcommand{\ppqewx}{\left(\ppq \circ \eta\right)(W,X)}
\newcommand{\ppqedx}{\left(\ppq \circ \eta\right)(D,X)}

\newcommand{\al}{\alpha}

\newcommand{\be}{\beta}
\newcommand{\floor}[1]{\left\lfloor #1 \right\rfloor}

\newcommand{\NQ}{\mathscr{N}(Q)}
\newcommand{\N}[1]{\mathscr{N}( #1 )}
\newcommand{\Nk}[2]{\mathscr{N}_{#2}( #1 )} 
\newcommand{\Nkt}[3]{\mathscr{N}_{#2}^{ \Rmnum{#3} }( #1 )} 

\newcommand{\NAPIQ}{\mathscr{N}^{I}(Q)}
\newcommand{\NAPIIQ}{\mathscr{N}^{II}(Q)}
\newcommand{\NAPI}[1]{\mathscr{N}^{I}( #1 )}
\newcommand{\NAPII}[1]{\mathscr{N}^{II}( #1 )}
\newcommand{\NAPIk}[2]{\mathscr{N}^{I}_{ #2 }( #1 )}
\newcommand{\NAPIIk}[2]{\mathscr{N}^{II}_{ #2 }( #1 )}
\newcommand{\NAPIkm}[3]{\mathscr{N}^{I}_{{ #2 }, { #3 } }( #1 )}
\newcommand{\NAPIIkm}[3]{\mathscr{N}^{II}_{{ #2 }, { #3 }}( #1 )}
\newcommand{\NAPIkmr}[4]{\mathscr{N}^{I}_{{ #2 }, { #3 }, { #4 } }( #1 )}
\newcommand{\NAPIIkmr}[4]{\mathscr{N}^{II}_{{ #2 }, { #3 }, { #4 } }( #1 )}

\newcommand{\DNQ}{\mathscr{DN}(Q)}
\newcommand{\DN}[1]{\mathscr{DN}( #1 )} 

\newcommand{\RNQ}{\mathscr{RN}(Q)}
 
\newcommand{\RNk}[2]{\mathscr{RN}_{#2}( #1 )} 

\newcommand{\RNAPIQ}{\mathscr{RN}^{I}(Q)}
\newcommand{\RNAPIIQ}{\mathscr{RN}^{II}(Q)}

\newcommand{\Nc}[1]{\mathbb{N}_{ #1 }}

\newcommand{\UPAmri}{\Upsilon_{Q,\mathscr{A}_{m,r}}^{-1}}
\newcommand{\LAAmr}{\Lambda_{\mathscr{A}_{m,r}}(Q)}

\newcommand{\qnk}{Q_n^{(k)}}
\newcommand{\pnk}{P_n^{(k)}}

\newcommand{\mr}{\mathscr{MR}}
\newcommand{\mrmr}{(m,r) \in \mathscr{MR}}

\begin{document}
\date{Revised \today}

\markboth{B. Li and B. Mance}{Number theoretic applications of a class of Cantor series fractal functions, II}

%
\catchline{}{}{}{}{}
%

\title{Number theoretic applications of a class of Cantor series fractal functions, II}

\author{Brian Li}
\address{Department of Mathematics, The Ohio State Universtiy, 231 West 18th Avenue\\
Columbus, Ohio 43210-1174, USA\\
\email{brianua13@yahoo.com} }

\author{Bill Mance}
\address{Department of Mathematics,  University of North Texas, General Academics Building 435\\
 1155 Union Circle \#311430, Denton, Texas 76203-5017, USA\\
\email{mance@unt.edu}  }

\maketitle


\begin{abstract}
It is well known that all numbers that are normal of order $k$ in base $b$ are also normal of all orders less than $k$.   Another basic  fact is that every real number is normal in base $b$ if and only if it is simply normal in base $b^k$ for all $k$.  
This may be interpreted to mean that a number is normal in base $b$ if and only if all blocks of digits occur with the desired relative frequency along every infinite arithmetic progression.  We reinterpret these theorems for the $Q$-Cantor series expansions and show that they are no longer true in a particularly strong way.  The main theoretical result of this paper will be to reduce the problem of constructing normal numbers with certain pathological properties to the problem of solving a system of Diophantine relations. 

\end{abstract}





\section{Introduction}

\subsection{Motivation}


We recall the modern definition of a normal number.

\begin{definition}\labd{normal}
A real number $x$ is {\it normal of order $k$ in base $b$} if all blocks of digits of length $k$ in base $b$ occur with relative frequency $b^{-k}$ in the $b$-ary expansion of $x$.  $x$ is {\it simply normal in base $b$} if it is normal of order $1$ in base $b$ and $x$ is {\it normal in base $b$} if it is normal of order $k$ in base $b$ for all natural numbers $k$.
\end{definition}

It is well known that \'{E}. Borel \cite{BorelNormal} was the first mathematician to study normal numbers.  In 1909 he gave the following definition.

\begin{definition}\labd{normalBorel}
A real number $x$ is {\it normal in base $b$} if each of the numbers $x, bx, b^2x,\cdots$ is simply normal (in the sense of \refd{normal}), in each of the bases $b, b^2, b^3, \cdots$ 
\end{definition}

\'{E}. Borel proved that Lebesgue almost every real number is normal, in the sense of \refd{normalBorel}, in all bases.  In 1940, S. S. Pillai \cite{Pillai2} simplified \refd{normalBorel} by proving that
\begin{thrm}\labt{Pillai}
For $b \geq 2$, a real number $x$ is normal in base $b$ if and only if it is simply normal in each of the bases $b, b^2, b^3, \cdots$.
\end{thrm}
\reft{Pillai} was improved in 1951 by I. Niven and H. S. Zuckerman \cite{NivenZuckerman} who proved
\footnote{
A simpler proof of \reft{Pillai} was given by J. E. Maxfield in \cite{MaxfieldPillai}.  J. W. S. Cassels gave a shorter proof of \reft{NivenZuckerman} in \cite{CasselsNZ}.
}
\begin{thrm}\labt{NivenZuckerman}
\refd{normal} and \refd{normalBorel} are equivalent.
\end{thrm}

 It should be noted that both of these results require some work to establish, but were assumed without proof by several authors.  For example, M. W. Sierpinski assumed \reft{Pillai} in \cite{Sierpinski} without proof.  Moreover, D. G. Champernowne \cite{Champernowne}, A. H. Copeland and P. Erd\H os \cite{CopeErd}, and other authors took \refd{normal} as the definition of a normal number before it was proven that \refd{normal} and \refd{normalBorel} are equivalent.  More information can be found in Chapter 4 of the book of Y. Bugeaud \cite{BugeaudBook}.

The following theorem was proven by H. Furstenberg in his seminal paper ``Disjointness in Ergodic Theory, Minimal Sets, and a Problem in Diophantine Approximation''\cite{FurstenbergDisjoint} on page 23 as an application of disjointness to stochastic sequences.
\begin{thrm}\labt{basebnormalii}
Suppose that $x=d_0.d_1d_2\cdots$ is the $b$-ary expansion of $x$.  Then $x$ is normal in base $b$ if and only if for all natural numbers $m$ and $r$ the real number $0.d_{r}d_{m+r}d_{2m+r}d_{3m+r}\cdots$ is normal in base $b$.
\end{thrm}
H. Furstenberg mistakenly claimed
\footnote{See Appendix}
 that this provided an alternate proof of \reft{NivenZuckerman}, but actually proved that an entirely different definition of normality is equivalent to \refd{normal}.  We will say that $x$ is {\it AP normal of type I in base $b$} if $x$ satisfies \refd{normalBorel} and {\it AP normal of type II in base $b$} if $x$ satisfies the notion introduced in \reft{basebnormalii}.  Thus, for numbers expressed in base $b$
$$
\hbox{normality }\Leftrightarrow \hbox{ AP normality of type I } \Leftrightarrow \hbox{ AP normality of type II}.
$$


The $Q$-Cantor series expansion, first studied by G. Cantor in \cite{Cantor},\footnote{G. Cantor's motivation to study the Cantor series expansions was to extend the well known proof of the irrationality of the number $e=\sum 1/n!$ to a larger class of numbers.  Results along these lines may be found in the monograph of J. Galambos \cite{Galambos}.} 
is a natural generalization of the $b$-ary expansion. 
Let $\mathbb{N}_k:=\mathbb{Z} \cap [k,\infty)$.  
If $Q \in \NN$, then we say that $Q$ is a {\it basic sequence}.
Given a basic sequence $Q=(q_n)_{n=1}^{\infty}$, the the {\it $Q$-Cantor series expansion} of a real $x$ in $\mathbb{R}$ is the (unique)\footnote{Uniqueness can be proven in the same way as for the $b$-ary expansions.} expansion of the form
\begin{equation} \labeq{cseries}
x=E_0+\sum_{n=1}^{\infty} \frac {E_n} {q_1 q_2 \ldots q_n},
\end{equation}
where $E_0=\floor{x}$ and $E_n$ is in $\{0,1,\ldots,q_n-1\}$ for $n\geq 1$ with $E_n \neq q_n-1$ infinitely often. We abbreviate \refeq{cseries} with the notation $x=E_0.E_1E_2E_3\ldots$ w.r.t. $Q$.
Clearly, the $b$-ary expansion is a special case of \refeq{cseries} where $q_n=b$ for all $n$.  If one thinks of a $b$-ary expansion as representing an outcome of repeatedly rolling a fair $b$-sided die, then a $Q$-Cantor series expansion may be thought of as representing an outcome of rolling a fair $q_1$ sided die, followed by a fair $q_2$ sided die and so on.

The authors feel that the equivalence of \refd{normal} and \refd{normalBorel} and other similar ones is a far more delicate topic than is typically assumed.  
The core of E. Borel's definition is that a number is normal in base $b$ if blocks of digits occur with the desired relative frequency along all arithmetic progressions.  
We will extend \refd{normalBorel} to the $Q$-Cantor series expansions and call it {\it AP $Q$-normality of type I}.  P. Laffer has already studied a similar definition in \cite{Laffer}, but made no comparison with the definition of $Q$-normality that naturally arises as an extension of \refd{normal}.  We will prove that the analogous extension of \reft{NivenZuckerman} to the $Q$-Cantor series expansions will no longer hold.  

We will refer to the natural extension of the concept introduced in \reft{basebnormalii} to the $Q$-Cantor series expansions as {\it AP $Q$-normality of type II}.  
We will show that for $Q$-Cantor series expansions that while $Q$-normality is implied by AP $Q$-normality of type I, the bigger picture of the relationship between these notions  is far more complex than for the $b$-ary expansions.\footnote{Theorem 2.5 in \cite{Mance4} gives that the set of AP $Q$-normal numbers of type I is contained in the set of $Q$-normal numbers.}

We also note the following trivial theorem about normal numbers in base $b$.

\begin{thrm}\labt{baseborders}
If $x$ is normal of order $k$ in base $b$, then $x$ is normal of orders $1, 2, \cdots, k-1$ in base $b$.
\end{thrm}
Perhaps much more suprising than the lack of a simple extension of \reft{NivenZuckerman} and \reft{basebnormalii} to the $Q$-Cantor series expansion is that for all $k$, we give an example of a computable basic sequence $Q$ and a computable real number $x$ that is $Q$-normal of order $k$, but not $Q$-normal of orders $1, 2, \cdots, k-1$.  An example is given in \reft{normalorderknotlower}.

In order to prove these theorems, we must extend the constructions given by the second author in \cite{Mance} to construct numbers that are AP $Q$ normal of types I and II. 
A basic sequence $Q=(q_n)$ is {\it infinite in limit} if $q_n \to \infty$.
Suppose that $M=(m_t)_t$ is an increasing sequence of positive integers.  Let $N_{M,n}^Q(B,x)$ be the number of occurrences of the block $B$ at positions $m_t$ for $m_t \leq n$ in the $Q$-Cantor series expansion of $\{x\}$ and let $N_n^Q(B,x)=N_{(t),n}^Q(B,x)$.
 If $x=E_0.E_1E_2 \cdots$ w.r.t. {P}, then put
$$
\ppq(x):=\sum_{n=1}^\infty \frac {\min(E_n,q_n-1)} {\q{n}}.
$$
The functions $\ppq$ and their properties
\footnote{See \cite{ppq1} for an overview of properties such as continuity and multifractal analysis of $\ppq$.  There are many fractals associated with the functions $\ppq$, but they will not affect the results discussed in this paper.}
 will be of critical importance to our constructions and were the topic to the predecessor to this paper by the second author \cite{ppq1}.  The key property of these functions is the following theorem that was proven in \cite{ppq1}.
\begin{thrm}\labt{mainpsi}
\footnote{
The conclusions of \reft{mainpsi} sometimes do not hold without the requirement that $E_n<\min_{2 \leq r \leq j} (q_{r,n}-1)$ for infinitely many $n$.  For example, consider $p_n=3$ and 
\begin{displaymath}
q_n=\left\{ \begin{array}{ll}
2 & \textrm{if $n \equiv 0 \pmod{2}$}\\
3 & \textrm{if $n \equiv 1 \pmod{2}$}
\end{array} \right. .
\end{displaymath}
Let $x=7/8=0.\overline{21}$ w.r.t. $P$.  Then $\ppq(x)=1.\overline{0}$ w.r.t. $Q$ so $N_n^P((1),x)=\floor{n/2}$ while $N_n^Q((1),\ppq(x))=0$ for all $n$.
}
Suppose that $M=(m_t)$ is an increasing sequence of positive integers and  $Q_1=(q_{1,n}), Q_2=(q_{2,n}),\cdots, Q_j=(q_{j,n})$ are basic sequences and infinite in limit.
  If $x=E_0.E_1E_2\cdots$ w.r.t $Q_1$ satisfies
$E_n<\min_{2 \leq r \leq j} (q_{r,n}-1)$ for infinitely many $n$, then for every block~$B$
$$
N_{M,n}^{Q_j}\left(B,\left(\psi_{Q_{j-1},Q_j} \circ \psi_{Q_{j-2},Q_{j-1}} \circ \cdots \circ \psi_{Q_1,Q_2}\right)(x)\right)
=N_{M,n}^{Q_1}(B,x)+O(1).
$$
\end{thrm}
We should note that not all constructions in the literature of normal numbers are of computable real numbers. For example, the construction by M. W. Sierpinski in \cite{Sierpinski} is not of a computable real number.  V. Becher and S. Figueira modified M. W. Sierpinski's work to give an example of a computable absolutely normal number in \cite{BecherFigueira}.  
 Since not every basic sequence is computable we face an added difficulty.  Moreover, many of the numbers constructed in \cite{ppq1} by using \reft{mainpsi} are not computable.  Thus, we will be careful to indicate which numbers we construct are computable.  No deep knowledge of computability theory will be used and any time we make such a claim there will exist a simple algorithm to compute the number under consideration to any degree of precision.

The main theoretical result of this paper will be to reduce the problem of constructing certain types of normal numbers to solving a system of Diophantine relations.  This reduction is given in \reft{normalorders} and \reft{APnormalorders} and relies on \reft{mainpsi}.  Both of these theorems are technical and require several definitions to state so we defer their formulations to \refs{mainresults}.
We finish the paper with several conjectures in \refs{conjectures}.

\subsection{Definitions}\labs{definitions}
A {\it block} is an ordered tuple of non-negative integers, a {\it block of length $k$} is an ordered $k$-tuple of integers, and {\it block of length $k$ in base $b$} is an ordered $k$-tuple of integers in $\{0,1,\ldots,b-1\}$.
For a given basic sequence $Q$, definitione\footnote{For the remainder of this paper, we will assume the convention that the empty sum is equal to $0$ and the empty product is equal to $1$.}
$$
Q_n^{(k)}:=\sum_{j=1}^n \frac {1} {q_j q_{j+1} \ldots q_{j+k-1}}.
$$

A. R\'enyi \cite{Renyi} definitioned a real number $x$ to be {\it normal} with respect to $Q$ if for all blocks $B$ of length $1$,
\begin{equation}\labeq{rnormal}
\lim_{n \rightarrow \infty} \frac {N_n^Q (B,x)} {Q_n^{(1)}}=1.
\end{equation}
If $q_n=b$ for all $n$ and we restrict $B$ to consist of only digits less than $b$, then \refeq{rnormal} is equivalent to {\it simple normality in base $b$}, but not equivalent to {\it normality in base $b$}. A basic sequence $Q$ is {\it $k$-divergent} if
$\lim_{n \rightarrow \infty} Q_n^{(k)}=\infty$.
$Q$ is {\it fully divergent} if $Q$ is $k$-divergent for all $k$ and {\it $k$-convergent} if it is not $k$-divergent.  

\begin{definition}\labd{1.7} A real number $x$ is {\it $Q$-normal of order $k$} if for all blocks $B$ of length $k$,
$$
\lim_{n \rightarrow \infty} \frac {N_n^Q (B,x)} {Q_n^{(k)}}=1.
$$
We let $\Nk{Q}{k}$ be the set of numbers that are $Q$-normal of order $k$.  $x$ is {\it $Q$-normal} if
$
x \in \NQ := \bigcap_{k=1}^{\infty} \Nk{Q}{k}.
$
\end{definition}

For $Q$ that are infinite in limit,
it has been shown that the set of all real $x$  that are $Q$-normal of order $k$ has full Lebesgue measure if and only if $Q$ is $k$-divergent \cite{Mance4}.  Early work in this direction has been done by A. R\'enyi \cite{Renyi}, T.  \u{S}al\'at \cite{Salat4}, and F. Schweiger \cite{SchweigerCantor}.  Therefore, if $Q$ is infinite in limit, then the set of all real $x$ that are $Q$-normal has full Lebesgue measure if and only if $Q$ is fully divergent.

Given a basic sequence $M$, we definitione the basic sequence $\Lambda_M(Q):=(q_{m_t})_{t=1}^\infty$.
If $x=E_0.E_1E_2\cdots$ w.r.t. $Q$, then let $\Upsilon_{Q,M}(x):=0.E_{m_1}E_{m_2}E_{m_3}\cdots$ w.r.t. $\Lambda_M(Q)$.  
Let $\mr:=\{(a,b) \in \mathbb{N} \times \mathbb{N} : 0 \leq b \leq a-1\}$.
For any $\mrmr$, let $\mathscr{A}_{m,r}:=(mt+r)_{t=0}^\infty$.  
Let $N_{n,m,r}^Q(B,x):=N_{\mathscr{A}_{m,r},n}(B,x)$ and ${N_{n,m,r}^Q}^\prime(B,x):=N_n^{\Lambda_{\mathscr{A}_{m,r}}(Q)}(B,\Upsilon_{Q,\mathscr{A}_{m,r}}(x))$.  The following definition is motivated by \refd{normalBorel} and \reft{basebnormalii}.
For $\mrmr$, let
$$
Q_{n,m,r}^{(k)} := \sum_{j=0}^{\lfloor \frac {n-r}{m} \rfloor} \frac {1}{q_{mj+r}q_{mj+r+1}...q_{mj+r+k-1}},
$$

\begin{definition}\labd{apnormal}
Let $\NAPIkmr{Q}{k}{m}{r}$ be the set of real numbers $x$ such that 
$$
\lim_{n \rightarrow \infty} \frac{N_{n,m,r}^Q(B,x)}{Q_{n,m,r}^{(k)}}=1
$$
for all blocks $B$ of length $k$.  Additionally, let $\NAPIkm{Q}{k}{m}:=\bigcap_{r=0}^{m-1} \NAPIkmr{Q}{k}{m}{r}$ and $\NAPIk{Q}{k}:=\bigcap_{k=1}^m \NAPIkm{Q}{k}{m}$.
A real number $x$ is {\it AP $Q$-normal of order $k$ of type I} if $x \in \NAPIk{Q}{k}$ and {\it AP $Q$-normal of type I} if $x \in \NAPIQ:=\bigcap_{k=1}^\infty \NAPIk{Q}{k}$.

Let 
$$
\NAPIIkmr{Q}{k}{m}{r}:=\UPAmri\left( \Nk{\LAAmr}{k}  \right)
$$
and $\NAPIIkm{Q}{k}{m}:=\bigcap_{r=0}^{m-1} \NAPIIkmr{Q}{k}{m}{r}$.
A real number $x$ is {\it AP $Q$-normal of order $k$ of type II} if $x \in \NAPIIk{Q}{k}:=\bigcap_{r=0}^{m-1} \NAPIIkmr{Q}{k}{k}{r}
$.
Let 
$$
\NAPII{Q}:=\bigcap_{k=1}^\infty  \bigcap_{\mrmr} \NAPIIkmr{Q}{k}{m}{r}
$$
 be the set of numbers that are {\it AP $Q$-normal of type II}.
\end{definition}

The sets $\NAPI{Q}$ and $\NAPII{Q}$ introduced in \refd{apnormal} give a natural extension of the notions of AP normality of type I and II given in the intro in \refd{normalBorel} and \reft{basebnormalii}.   These sets along with $\NAPIk{Q}{k}$ and $\NAPIIk{Q}{k}$ will be the main focus of this paper.

A basic sequence $Q$ is {\it $(k,m,r)$-divergent of type I (resp. type II)} if $\lim_{n \rightarrow \infty} Q_{n,m,r}^{(k)}=\infty$
(resp. $\lim_{n \rightarrow \infty} (\LAAmr)_n^{(k)}=\infty$).
$Q$ is {\it fully divergent of type I} if $Q$ is $(k,m,r)$-divergent of type I  for all $k$ and $\mrmr$.
$Q$ is {\it fully divergent of type II} if $Q$ is $(k,k,r)$-divergent of type II  for all $k$ and $r$.
Suppose that $Q$ is infinite in limit.  It can be proven using the method in \cite{Mance4} that $\NAPIkmr{Q}{k}{m}{r}$ (resp. $\NAPIIkmr{Q}{k}{m}{r}$) is a set of full Lebesgue measure if and only if $Q$ is $(k,m,r)$-divergent of type I (resp. type II).


\begin{definition}
\footnote{It is unknown how the sets $\RNQ$, $\RNAPIQ$, and $\RNAPIIQ$ are related except that $\RNAPIQ \subseteq \RNQ$.}
A real number $x$ is {\it $Q$-ratio normal of order $k$} (here we write $x \in \RNk{Q}{k}$) if for all blocks $B_1$ and $B_2$ of length $k$
$$
\lim_{n \to \infty} \frac {N_n^Q (B_1,x)} {N_n^Q (B_2,x)}=1.
$$
Additonally, $x$ is {\it $Q$-ratio normal} if
$
x \in \RNQ := \bigcap_{k=1}^{\infty} \RNk{Q}{k}$.  The sets $\RNAPIQ$ and $\RNAPIIQ$ are definitioned similarly to $\NAPIQ$ and $\NAPIIQ$.

\end{definition}

\subsection{Results}\labs{mainresults}

For the rest of this paper, given sequences of non-negative integers $l=(l_i)$ and $b=(b_i)$ with $b_i \geq 2$ and sequences blocks of digits $X=(X_i)$, we write  $L_i:=\left| X_1^{l_1} \ldots X_i^{l_i}\right|=l_1 |X_1|+\ldots+l_i |X_i|$.  Additionally, we let $\Gamma(l,b,X):=(\gamma_n)_{n=1}^{\infty}$, where $\gamma_n:=b_i$ for $L_{i-1} < n \leq L_i$ and $\eta(l,b,X):=\sum_{n=1}^{\infty} \frac {E_n} {\gamma_1 \cdots \gamma_n}$, where $(E_1,E_2,\ldots)=X_1^{l_1} X_2^{l_2} X_3^{l_3} \cdots$.  

For a compact metric space $X$, let $\mathscr{M}(X)$ be the collection of all Borel probability measures on $X$.
Given $\mu \in \measNNz$ and $B=(b_1,\cdots,b_k) \in \Nc{0}^k$, we write
$$
[B]:=\left\{\omega=(\omega_1,\omega_2,\cdots) \in \Nc{0}^{\mathbb{N}} : \omega_j=b_j \forall j \in [1,k]       \right\} \hbox{ and }
\mu(B):=\mu\left([B]\right).
$$
A block of digits $Y$ is~{\it $(\epsilon,k,\mu)$-normal} if for all blocks $B$ of length $m \leq k$, we have 
$$(1-\epsilon)|Y|\mu(B) \leq N(B,Y) \leq (1+\epsilon) |Y| \mu(B).$$
A measure $\mu \in \measNNz$ is {\it $(v,b)$-uniform} if for all $k$ and blocks $B$ of length $k$ in base $v \leq b$, we have $\mu(B)=b^{-k}$.
A {\it block friendly family(BFF)}, $W$, is a sequence of $6$-tuples $((l_i,b_i,v_i,\epsilon_i,k_i,\mu_i))_{i=1}^{\infty}$ with non-decreasing sequences of non-negative integers $(l_i)_{i=1}^{\infty}$, $(b_i)_{i=1}^{\infty}$, $(v_i)_{i=1}^{\infty}$, and $(k_i)_{i=1}^{\infty}$ for which $b_i \geq 2$, $b_i \rightarrow \infty$ and $v_i \rightarrow \infty$, such that $(\mu_i)_{i=1}^{\infty} \in \measNNz^{\mathbb{N}}$ is a sequence of $(v_i,b_i)$-uniform measures and $(\epsilon_i)_{i=1}^{\infty}$ strictly decreases to $0$.
Let $R(W):=[1,\lim_{i \to \infty} k_i] \cap \mathbb{N}$.
If $(X_i)_{i=1}^{\infty}$ is a sequence of blocks such that $|X_i|$ is non-decreasing and $X_i$ is $(\epsilon_i,k_i,\mu_i)$-normal, then $(X_i)_{i=1}^{\infty}$ is   {\it $W$-good} if for all $k$ in $R$,
\begin{align}\labeq{good1}
\frac {b_i^k} {\epsilon_{i-1}-\epsilon_i} &= o(|X_i|) ;\\
\labeq{good2}
\frac {l_{i-1}} {l_i} \cdot \frac {|X_{i-1}|} {|X_i|}&=o(i^{-1}b_i^{-k});\\
\labeq{good3}
\frac {1} {l_i} \cdot \frac {|X_{i+1}|} {|X_i|}&=o(b_i^{-k}).
\end{align}
We will write $\Gamma(W,X):=\Gamma(l,b,X)$ and $\eta(W,X):=\eta(l,b,X)$.  The following is the main theorem in the paper \cite{Mance} of the second author.\footnote{Our statement of \reft{thm3.1} and the preceding definitions have been altered to be more concisely stated than they were in \cite{Mance}.  We have also removed some unnecessary hypotheses.}
\begin{thrm}\labt{thm3.1}
Let $W=((l_i,b_i,v_i,\epsilon_i,k_i,\mu_i))_{i=1}^{\infty}$ be a BFF and suppose that  $X=(X_i)_{i=1}^{\infty}$ is $W$-good. 
Then\footnote{Clearly, if $k_i \rightarrow \infty$, then $\eta(W,X) \in \N{\Gamma(W,X)}$.}
$$\eta(W,X) \in \bigcap_{k \in R(W)} \Nk{\Gamma(W,X)}{k}.$$
If $W$ and $X$ are computable, then $\eta(W,X)$ and $\Gamma(W,X)$ are computable.
\end{thrm}
It is not difficult to see that $\Gamma(W,X)$ must always be fully divergent by using \refeq{good3}. We wish to extend \reft{thm3.1} to deal with $\NAPIkmr{Q}{k}{m}{r}$ and $\NAPIIkmr{Q}{k}{m}{r}$.

A block of digits $X$ is {\it $(\epsilon, k, m, \mu)$-normal of type I} (resp. {\it $(\epsilon, k, m, \mu)$-normal of type II}) if for all blocks $B$ of length $k^\prime \le k$ and for all $m^\prime \le m$ and $0 < r \le m^\prime$, $N_{|X|,m^\prime,r}(B,X)$ (resp. $N_{|X|,m^\prime,r}^\prime(B,X)$) is in the closed interval
$$
\left[\mu(B)\left\lceil\frac {|X|-r}{m'}\right\rceil(1-\epsilon), \mu(B)\left\lceil\frac {|X|-r}{m'}\right\rceil(1+\epsilon)\right].
$$
For a sequence of $7$-tuples $$D=((l_i,b_i,v_i,\epsilon_i,k_i,\mu_i,m_i))_{i=1}^\infty,$$ let $\hbox{BFF}(D):=((l_i,b_i,v_i,\epsilon_i,k_i,\mu_i))_{i=1}^\infty$.
An {\it APBFF} is a sequence of $7$-tuples $D=((l_i,b_i,v_i,\epsilon_i,k_i,\mu_i,m_i))_{i=1}^\infty$ where $\hbox{BFF}(D)$ is a BFF and $(m_i)$ is a sequence of non-decreasing integers.
Let $D$ be an APBFF and set $R(D):=R(\hbox{BFF}(D))$ and $S(D):=[1,\lim_{i \to \infty} m_i] \cap \mathbb{N}$. 
A sequence $(X_i)_{i=1}^{\infty}$ of $(\epsilon_i,k_i,m_i, \mu_i)$-normal of type I (resp. type II) blocks of non-decreasing length is said to be $D$-good of type I (resp. type II) if for all $k$ and $m$ in $R$, the conditions \refeq{good1}, \refeq{good2}, and \refeq{good3} hold.
For an APBFF $D=((l_i,b_i,v_i,\epsilon_i,k_i,\mu_i,m_i))_{i=1}^\infty$ and sequence of blocks $X=(X_i)$, we let $\Gamma(D,X):=\Gamma(\hbox{BFF}(D),X)$ and $\eta(D,X):=\eta(\hbox{BFF}(D),X)$.

\begin{thrm}\labt{APmain}
\footnote{
We note that \reft{APmain} constructs numbers that satisfy a condition that might be stronger than being contained in $\NAPIQ$ if $\lim_{i \to \infty} \min(k_i,m_i)=\infty$ and $(X_i)$ is $D$-good of type I.  That is,
$$
\eta(D,X) \in \bigcap_{k=1}^\infty \bigcap_{m=1}^\infty \NAPIkm{\Gamma(D,X)}{k}{m} \subseteq \NAPI{\Gamma(D,X)}.
$$
It seems likely that there exists a basic sequence $Q$ where $\bigcap_{k=1}^\infty \bigcap_{m=1}^\infty \NAPIkm{Q}{k}{m}$ is strictly contained in $\NAPI{Q}$, but the theorems proven in this paper do not seem to be strong enough to establish this.  See \refj{APabnormal}.
}
Let $D$ be an APBFF and $(X_i)_{i=1}^{\infty}$ a $D$-good sequence of type I (resp. type II).  Then
$$
\eta(D,X) \in \bigcap_{k \in R(D)} \bigcap_{m \in S(D)} \NAPIkm{\Gamma(D,X)}{k}{m} \ 
\left( \hbox{resp. }  \bigcap_{k \in R(D)} \bigcap_{m \in S(D)} \NAPIIkm{\Gamma(D,X)}{k}{m} \right).
$$
If $k_i \rightarrow \infty$ and $m_i \rightarrow \infty$, then $\eta(D,X) \in \NAPI{\Gamma(D,X)} \left(\hbox{resp. } \NAPII{\Gamma(D,X)} \right)$.
\end{thrm}

We will prove \reft{APmain} in \refs{APmain}.
For the rest of this paper, let $C_{b,w}$ be the block formed by concatentating all the blocks of length $w$ in base $b$ in lexicographic order.
\footnote{For example, $C_{2,2} = (0,0)(0,1)(1,0)(1,1) = (0,0,0,1,1,0,1,1)$.}
Many of our examples of normal numbers will be built up by concatenating blocks of this form.  We will study properties of these blocks in \refs{cbw}.

\begin{thrm} \labt{construction}
Let $t \in \mathbb{N}_2$ and put $l_i=0$ and $X_i=(0)$ for $i <6$.
For $i \ge 6$, let $X_{i,t} = C_{it,i!}$, $b_{i,t} = it$, $l_i = 3^{i!} \cdot (i+1)^{i! \cdot i}$, $\epsilon_i = \frac{1}{i}$, $k_i = i$, $v_{i,t} = it$, $m_i = i$, $\mu_i = \lambda_i$, and $D_t=\{(l_i,b_{i,t},v_{i,t},\epsilon_i,k_i,\mu_i,m_i)\}_{i=1}^{\infty}$.  Let $X_t=(X_{i,t})$, $\zeta_t=\eta(D_t,X_t)$, and $R_t=\Gamma(D_t,X_t)$. Then $\zeta_t \in \NAPI{R_t} \cap \NAPII{R_t}$.
\end{thrm}

\reft{construction} is proven in \refs{cbw}.  For $P=(p_n) \in \NN$, 
set $\Xi(P,(c_0,\cdots,c_{t-1}),d):=(\xi_n)$, where 
\begin{align*}
\xi_n:=\left\{ \begin{array}{ll}
\max(2,c_0^{-1}p_n) & \textrm{if $n \equiv 0 \pmod{d}$}\\
\max(2,c_1^{-1}p_n) & \textrm{if $n \equiv 1 \pmod{d}$}\\
\cdots\\
\max(2,c_{t-1}^{-1}p_n) & \textrm{if $n \equiv t-1 \pmod{d}$}\\
2^n p_n & \textrm{if $n \equiv t,t+1,\cdots,d-1 \pmod{d}$}
\end{array} \right. .
\end{align*}
We note that $Q$ is a basic sequence if  for $j \in [0,t-1]$ we have $c_j=\frac {\al_j} {\be_j}$ in lowest terms for $\al_j,\be_j \in \mathbb{Z}$ where $\al_j | p_n$ for all $(j,n) \in [0,t-1] \times \mathbb{N}$.  $Q$ is infinite in limit if and only if $P$ is.  If $P$ is of the form $P=\Gamma(W,X)$, then $Q$ is fully divergent of types I and II.

\begin{thrm}\labt{normalorders}
Let $W=((l_i,b_i,v_i,\epsilon_i,k_i,\mu_i))_{i=1}^{\infty}$ be a BFF. Suppose that $(X_i)$ is $W$-good and $b_{i+1}/b_i \to 1$.
Let $t \geq 2$ be an integer and suppose that $A \cupdot B=\{1,\cdots,t\}$ and $\lim_{i \to \infty} k_i \geq t$.  Suppose that there is a solution 
$$(c_0,\dots,c_{t-1},d)=\left(\frac {\al_0} {\be_0},\cdots, \frac {\al_{t-1}} {\be_{t-1}},d\right) \in \mathbb{Q}_{>0}^t \times \mathbb{N}_{t+1},$$ with $\al_i$ and $\be_i$ relatively prime for all $i$, of the $t \times (t+1)$ system of Diophantine relations given by
\begin{equation}\labeq{normalorders1}
\sum_{j=0}^{t-k} c_jc_{j+1}\cdots c_{j+k-1} = d
\end{equation}
if $k \in A$ and
\begin{equation}\labeq{normalorders2}
\sum_{j=0}^{t-k} c_jc_{j+1}\cdots c_{j+k-1} \neq d
\end{equation}
if $k \in B$ for all $k \in \{1,\cdots,t\}$.
Additionally, assume that $\al_j | p_n$ for all $j \in [0,t-1]$ and all $(n,j) \in \mathbb{N} \times [0,t-1]$. 
Put $P=\Gamma(W,X)$ and $Q=\Xi(P,(c_0,\cdots,c_{t-1}),d)$.  Then $Q$ is fully divergent and
$$
\ppqewx \in \bigcap_{k \in A} \Nk{Q}{k} \backslash \bigcap_{j \in B} \Nk{Q}{j} \neq \emptyset.
$$

For example, if $t=3, A=\{2,3\}$, and $B=\{1\}$, then we are considering the system
\begin{align*}
c_0+c_1+c_2&\neq d\\
c_0c_1+c_1c_2&=d\\
c_0c_1c_2&=d,
\end{align*}
which has a solution $(c_0,c_1,c_2,d)=(2,1,2,4)$.  Thus, $\ppqewx \in \Nk{Q}{2} \cap \Nk{Q}{3} \backslash \Nk{Q}{1}$ when $W$ is a BFF and $X$ is a $W$-good sequence.  However, not all such systems have solutions.

\end{thrm}

For simplicity, we stated \reft{normalorders} in such a way that only applies to numbers constructed by using \reft{thm3.1}.  A slightly more general theorem may be stated that applies to $P$-normal numbers where $P$ is constant for long stretches.  However, this greater generality does not seem to allow for any more interesting examples of normal numbers to be constructed.  We now state a similar theorem to \reft{normalorders} that may be used to analyze the sets $\NAPIkmr{Q}{k}{m}{r}$ and $\NAPIIkmr{Q}{k}{m}{r}$.

\begin{thrm}\labt{APnormalorders}
Let $D=\{(l_i,b_i,v_i,\epsilon_i,k_i,\mu_i,m_i)\}_{i=1}^{\infty}$  be an APBFF.  Suppose that $(X_i)$ is $D$-good of type~I  and $b_{i+1}/b_i \to 1$. Let $t \in [2, \lim_{i \to \infty} \min(k_i,m_i)] \cap \mathbb{N}$, $k \in [1,\lim_{i \to \infty} k_i] \cap \mathbb{N}$, $m \in [1,\lim_{i \to \infty} m_i] \cap \mathbb{N}$, and $r \in [0,m-1] \cap \mathbb{N}$. 
Put $u=\lcm(t,m)$ and suppose that  $(c_0,\dots,c_{t-1},d)=\left(\frac {\al_0} {\be_0},\cdots, \frac {\al_{t-1}} {\be_{t-1}},d\right) \in \mathbb{Q}_{>0}^t \times u\mathbb{N}_2$, with $\al_i$ and $\be_i$ relatively prime for all $i$, $\al_j | p_n$ for all $(n,j) \in \mathbb{N} \times \{0,\cdots,t-1\}$. 
Put $P=\Gamma(D,X)$ and $Q=\Xi(P,(c_0,\cdots,c_{t-1}),d)$. Then $Q$ is fully divergent of types I and II. 
Additionally,
\begin{align}\labeq{APIorders1}
&\sum_{j=0}^{\floor{\frac {t-k-r} {m}}} c_{r+jm}c_{r+jm+1}\cdots c_{r+jm+k-1} = d/m
\implies \ppqedx \in \Nkt{Q}{k,m,r}{1};\\
\labeq{APIorders2}
&\sum_{j=0}^{\floor{\frac {t-k-r} {m}}} c_{r+jm}c_{r+jm+1}\cdots c_{r+jm+k-1} \neq d/m
\implies \ppqedx \notin \Nkt{Q}{k,m,r}{1}.
\end{align}

Suppose, instead, that $(X_i)$ is $D$-good of type II.  
Then
\begin{align}\labeq{APIIorders1}
&\sum_{j=0}^{\floor{\frac {t-r-1} {m}}-k+1} c_{r+jm}c_{r+(j+1)m}\cdots c_{r+(j+k-1)m} = d/m
\implies \ppqedx \in \Nkt{Q}{k,m,r}{2};\\
\labeq{APIIorders2}
&\sum_{j=0}^{\floor{\frac {t-r-1} {m}}-k+1} c_{r+jm}c_{r+(j+1)m}\cdots c_{r+(j+k-1)m} \neq d/m
\implies \ppqedx \notin \Nkt{Q}{k,m,r}{2}.
\end{align}
\end{thrm}

\begin{thrm}\labt{normalorderknotlower}
Let $t \in \mathbb{N}_2$ and $\zeta_t$ and $R_t$ be definitioned as in \reft{construction}.  Set $Q=\Xi(R_t,(t!,1,1,\cdots,1),t!)$.  
Then
$$\psi_{R_t,Q}(\zeta_t) \in \Nk{Q}{t} \backslash \left(\bigcup_{j=1}^{t-1} \Nk{Q}{j} \cup \bigcup_{j=1}^t \Nkt{Q}{j}{1}\cup \bigcup_{j=1}^t \Nkt{Q}{j}{2} \right) \neq \emptyset$$
and $\psi_{R_t,Q}(\zeta_t)$ is computable.
\end{thrm}

\reft{normalorderknotlower} shows that a number that is $Q$-normal of order $k$ need not be $Q$-normal of any of the orders $1,2,\cdots,k-1$.  At the same time, it also shows that members of $\Nk{Q}{k}$ need not be in $\Nkt{Q}{j}{1}$ or $\Nkt{Q}{j}{2}$ for $j=1,2,\cdots,k$.
In contrast to AP normality of type I, there exists a basic sequence $Q$ where $\Nkt{Q}{k}{2} \backslash \Nk{Q}{k} \neq \emptyset$.  \reft{ap2notnormal} gives an example of such a $Q$ and a member of this set.

\begin{thrm}\labt{ap2notnormal}
Let $k \in \mathbb{N}_2$ and put $t=2k^2$.  Let $\zeta_t$ and $R_t$ be definitioned as in \reft{construction}.  
Set $c_0=c_1=\cdots=c_{k-1}=2k$ and $c_{k^2}=c_{k^2+1}=\cdots=c_{k^2+(k-1)}=2k$.  For all other $n$, we set $c_n=1$.
If $Q=\Xi(R_t,(c_0,c_1,\cdots,c_{t-1}),2k^2(k+1))$, then
$$\psi_{R_t,Q}(\zeta_t) \in \Nkt{Q}{k}{2} \backslash \Nk{Q}{k} \neq \emptyset$$
and $\psi_{R_t,Q}(\zeta_t)$ is computable.
\end{thrm}

\reft{normalorderknotlower} and \reft{ap2notnormal} follow from \reft{normalorders} and \reft{APnormalorders}.  We give a proof of \reft{ap2notnormal} in \refs{ap2notnormal}.  \reft{normalorderknotlower} is proven similarly.
Theorem~2.5 in \cite{Mance4}, \reft{normalorderknotlower}, and \reft{ap2notnormal} immediately imply the following main result.

\begin{figure}
\begin{center}
\begin{tikzpicture}[>=stealth',shorten >=1pt,node distance=2.8cm,on grid,initial/.style    ={}]
  \node[state]          (NQ)                        {$\mathsmaller{\Nk{Q}{k}}$};
  \node[state]          (NQAPII) [right =of NQ]    {$\mathsmaller{\NAPIIk{Q}{k}}$};
  \node[state]          (NQAPI) [above left=of NQAPII]    {$\mathsmaller{\NAPIk{Q}{k}}$};
\tikzset{mystyle/.style={->,double=black}} 
\tikzset{every node/.style={fill=white}} 
\path       (NQAPI)     edge [mystyle]     (NQ);
\end{tikzpicture}
\caption{}
\labf{APrelated}
\end{center}
\end{figure}

\begin{thrm}
When $k \geq 2$, \reff{APrelated} describes the complete containment relation between the sets $\Nk{Q}{k}, \NAPIk{Q}{k}$, and $\NAPIIk{Q}{k}$ for $Q$ that are infinite in limit and fully divergent of types I and II.
\end{thrm}

It should be noted that I. Niven and H. S. Zuckerman proved, but did not state, in \cite{NivenZuckerman} that for $b$-ary expansions, every number that is normal of order $k$ must also be AP normal of order $k$ of type I.  Thus, \reff{APrelated} shows a difference between $Q$-Cantor series expansions and $b$-ary expansions.  Further possible differences are discussed in \refj{APabnormal}.


\reft{mainpsi} has the following immediate corollary that we will use in the proofs of \reft{APnormalorders} and \reft{rationormal}.
\begin{cor}\labc{mainpsi2}
Suppose that $M=(m_t)$ is an increasing sequence of positive integers, $Q_1=(q_{1,n}), Q_2=(q_{2,n}),\cdots, Q_j=(q_{j,n})$ are basic sequences and infinite in limit.  If $x=E_0.E_1E_2\cdots$ w.r.t $Q_1$ satisfies
$E_{m_t} < \min_{1 \leq r \leq j} (q_{r,m_t}-1)$ for infinitey many $t$, then for every block~$B$
\begin{align*}
N_n^{\Lambda_M(Q_j)}\left(B,\left(\psi_{\Lambda_M(Q_{j-1}),\Lambda_M(Q_j)} \circ \cdots \circ \psi_{\Lambda_M(Q_1),\Lambda_M(Q_2)}\right)\left(\Upsilon_{Q_1,M}(x)\right)\right)\\
=N_n^{\Lambda_M(Q_1)}(B,\left(\Upsilon_{Q_1,M}(x)\right))+O(1).
\end{align*}
\end{cor}

\begin{thrm}\labt{rationormal}
Let $Q \in \NN$ be infinite in limit. Let $\zeta_1$ and $R_1$ be as in \reft{construction}. Then 
$$\psi_{R_1,Q}(\zeta_1) \in \RNAPIQ \cap \RNAPIIQ \subseteq \RNQ.$$
  If $Q$ is a computable sequence, then $\psi_{R_1,Q}(\zeta_1)$ is a computable real number.
\end{thrm}

\reft{rationormal} follows directly from \reft{mainpsi}, \reft{construction}, and \refc{mainpsi2}.

\section{Proofs of theorems}

\subsection{Proof of \reft{APmain}}\labs{APmain}
We prove the first part of \reft{APmain} that gives conditions for constructing members of $\NAPIkm{Q}{k}{m}$.  The second half is proven similarly with nearly identical lemmas.
In this subsection, we fix an APBFF $D=((l_i,b_i,v_i,\epsilon_i,k_i,\mu_i,m_i))_{i=1}^\infty$ and a  sequence of blocks $X=(X_i)$ that is $D$-good of type I.  Put $Q=(q_n)=\Gamma(D,X)$, $x=\eta(D,X)$, and let 
$m \in S(D)$.  We list some of the lemmas in this section without proof due to their similarity to lemmas leading towards the proof of the main theorem in \cite{Mance}.  However, we leave a few of the proofs to demonstrate some of the difference between the two.

Let $n \in \mathbb{N}$ be given and let $i=i(n)$ be the integer that satisfies $L_i < n \leq L_{i+1}$.  Let
$u$,  $\alpha$, and $\beta$ be integers such that $(\al,\be) \in [0,l_{i+1}] \times [0,|X_{i+1}|)$, and $u = n - L_i = \alpha|X_i| + \beta$. 

\begin{lem}\labl{firstlemma}
If $X_i$ is $(\epsilon_i,k_i,m_i,\mu_i)$-normal of type I, $k \le k_i$, and $B$ is a block of length $k$ in base $v_i$, then
\begin{align} \labeq{boundi}
(1-\epsilon_i)b_i^{-k}\left\lceil \frac{|X_i| - r}{m}\right\rceil &\le N_{m,r}(B,X_i) \le (1+\epsilon_i)b_i^{-k}\left\lceil \frac{|X_i| - r}{m}\right\rceil;\\
\labeq{boundi2}
(1-\epsilon_{i+1}) b_{i+1}^{-k}\alpha\left\lceil \frac{|X_{i+1}| - r}{m}\right\rceil &\le N_{u,m,r}(B,l_{i+1}X_{i+1})
\le (1+\epsilon_{i+1}) b_{i+1}^{-k}\alpha\left\lceil\frac{|X_{i+1}| - r}{m}\right\rceil + \left\lceil \frac{k\alpha}{m}\right\rceil + \left\lceil \frac{\beta}{m}\right\rceil.
\end{align}
\end{lem}
The next lemma follows by an application of \refl{firstlemma}
\begin{lem}\labl{secondlemma}
If $k \le k_i$ and $B$ is a block of length $k$ in base $v_i$,
\begin{equation} \labeq{boundn}
(1-\epsilon_i)b_i^{-k}l_i\left\lceil\frac{|X_i|-r}{m}\right\rceil + (1-\epsilon_{i+1})b_{i+1}^{-k}\alpha\left\lceil\frac{|X_{i+1}|-r}{m}\right\rceil \le N_{n,m,r}^Q(B,X) \le
\end{equation}
\begin{equation*}
\left\lceil\frac{L_{i-1}-r}{m}\right\rceil + (1+\epsilon_{i})b_{i}^{-k}l_i\left\lceil\frac{|X_{i}|-r}{m}\right\rceil + (1+\epsilon_{i+1})b_{i+1}^{-k}\alpha\left\lceil\frac{|X_{i+1}|-r}{m}\right\rceil + \left\lceil\frac{\beta}{m}\right\rceil + \left\lceil\frac{k\alpha}{m}\right\rceil + \left\lceil\frac{k}{m}\right\rceil(l_i+1).
\end{equation*}
\end{lem}
We  introduce a quantity $S_{n,m,r}^{(k)}$ to estimate $Q_{n,m,r}^{(k)}$.
\begin{equation*}
S_{n,m,r}^{(k)} = \sum_{j=1}^{i} b_j^{-k} \left\lceil\frac{l_j|X_j|-r}{m}\right\rceil + b_{i+1}^{-k}\left\lceil\frac{|u|-r}{m}\right\rceil
\end{equation*}
\begin{lem}\labl{thirdlemma}
$\lim_{n \to \infty} \frac {Q_{n,m,r}^{(k)}}{{S}_{n,m,r}^{(k)}} = 1$.
\end{lem}
\begin{proof} We will need to show that $\frac{{S}_{n,m,r}^{(k)} - Q_{n,m,r}^{(k)}}{{S}_{n,m,r}^{(k)}} \to 0$.
Let $s=\min\left\{t:k<\left\lceil\frac{|X_t|-r}{m}\right\rceil \right\}$.  
For $j \geq s$, set
\begin{equation*}
\bar{Q}_{j}^{(k)} = \left(\left(\frac{1}{b_j^k} + ... + \frac{1}{b_j^k}\right) + \left( \frac{1}{b_j^{k-1}b_{j+1}} + ... + \frac{1}{b_jb_{j+1}^{k+1}}\right)\right) = \frac{l_j|X_j|-(k-1)}{b_j^k} + \sum_{t=1}^{k-1}\frac{1}{b_j^{k-t}b_{j+1}^{t}}
\end{equation*}
and $\bar{S}_{j}^{(k)} = \frac{l_j|X_j|}{b_j^k}$.
We note that for $n$ large enough that $i(n)>s$
\begin{align*}
Q_{n,m,r}^{(k)} &= Q_{L_{s-1},m,r}^{(k)} + \sum_{\substack{j=s \\ j \equiv r \bmod m}}^i\bar{Q}_{j}^{(k)} + {\sum_{\substack{t = L_i + 1\\ t \equiv r\bmod m}}^n} \frac{1}{q_tq_{t+1}...q_{t+k+1}};\\
{S}_{n,m,r}^{(k)} &= {S}_{L_{s}-1,m,r}^{(k)} + \sum_{\substack{j=s \\ j \equiv r \bmod m}}^i\bar{S}_{j}^{(k)} +  {\sum_{\substack{t = L_i + 1\\ t \equiv r\bmod m}}^n} \frac{1}{q_tq_{t+1}...q_{t+k+1}}.
\end{align*}
We can observe that ${S}_{n,m,r}^{(k)} \ge Q_{n,m,r}^{(k)}$ since all the terms after $L_i + 1$ are the same, but the terms up to $L_i$ satisfy $\bar{S}_{j}^{(k)} \ge \bar{Q}_{j}^{(k)}$. Thus, ${S}_{n,m,r}^{(k)} - Q_{n,m,r}^{(k)} \ge 0$  when $i(n)>s$ and is an increasing function of $i=i(n)$, so
\begin{equation*}
{S}_{n,m,r}^{(k)} - Q_{n,m,r}^{(k)} \le {S}_{L_{i+1},m,r}^{(k)} - Q_{L_{i+1},m,r}^{(k)} = \left({S}_{L_{s-1},m,r}^{(k)} - Q_{L_{s}-1,m,r}^{(k)}\right) + \sum_{\substack{j=s \\ j \equiv r \bmod m}}^{i+1} \left( \bar{S}_{j}^{(k)} - \bar{Q}_{j}^{(k)} \right).
\end{equation*}
We now wish to estimate
$$
\sum_{\substack{j=s \\ j \equiv r \bmod m}}^{i+1} \left( \bar{S}_{j}^{(k)} - \bar{Q}_{j}^{(k)} \right).
$$
In the proof of Lemma 2.3 in \cite{Mance}, it is shown that
$\bar{S}_{j}^{(k)} - \bar{Q}_{j}^{(k)} < k$, so
\begin{equation*}
\sum_{\substack{j=s \\ j \equiv r \bmod m}}^{i+1} \left( \bar{S}_{j}^{(k)} - \bar{Q}_{j}^{(k)} \right) < \left\lceil \frac{i+2-s}{m} \right\rceil k.
\end{equation*}
Put $v = {S}_{L_{s-1},m,r}^{(k)} - Q_{L_{s-1},m,r}^{(k)}$.  Then
${S}_{n,m,r}^{(k)} - Q_{n,m,r}^{(k)} \le v + \left\lceil \frac{i+2-s}{m} \right\rceil k$
and
\begin{equation*}
\frac{{S}_{n,m,r}^{(k)} - Q_{n,m,r}^{(k)}}{{S}_{n,m,r}^{(k)}} < \frac{(v + k + \frac{2k}{m} - \frac{2s}{m}) + \frac{ki}{m}}{\frac{l_i|X_i|}{m}} = \frac{(mv + mk + 2k - 2s) + ki}{l_i|X_i|}.
\end{equation*}
We know $m$, $k$, and $s$ do not depend on $i$, and since $v$ depends only on $m$, $k$, and $s$, it also does not vary in terms of $i$. Thus, we know $l_i|X_i|$ 
$\frac{i}{l_i|X_i|} \to 0$  by  \refeq{good2}, so
\begin{equation*}
\frac{{S}_{n,m,r}^{(k)} - Q_{n,m,r}^{(k)}}{{S}_{n,m,r}^{(k)}} < \frac{(mv + mk + 2k - 2s) + ki}{l_i|X_i|} \to 0 \Rightarrow \lim_{n \to \infty} \frac {Q_{n,m,r}^{(k)}}{{S}_{n,m,r}^{(k)}} = 1.
\end{equation*}
\end{proof}
We will definitione two rational functions to estimate $\left|\frac{N_{n,m,r}^Q(B,X)}{Q_{n,m,r}^{(k)}} - 1\right|$. Let
\begin{align*}
f_i(w,z) &= \frac{\left({S}_{L_{i-1},m,r}+\epsilon_ib_i^{-k}l_i\left\lceil\frac{|X_i|-r}{m}\right\rceil\right) +w\left(\epsilon_{i+1}b_{i+1}^{-k}\left\lceil \frac{|X_{i+1}|-r}{m} \right\rceil \right)+z\left(\frac{b_{i+1}^{-k}}{m}\right)}{{S}_{L_i,m,r}^{(k)} + w\left(b_{i+1}^{-k}\left\lceil \frac{|X_{i+1}|-r}{m} \right\rceil \right)+z \left(\frac{b_{i+1}^{-k}}{m}\right)};\\
g_i(w,z) &= \frac{1}{{S}_{L_i,m,r}^{(k)} + w\left(b_{i+1}^{-k}\left\lceil \frac{|X_{i+1}|-r}{m} \right\rceil \right)+z \left(\frac{b_{i+1}^{-k}}{m}\right)}    \Bigg(\bigg(\frac{L_{i-1}-r}{m} + 3+ b_{i}^{-k}l_i + \epsilon_{i}b_{i}^{-k}l_i\left\lceil\frac{|X_{i}|-r}{m}\right\rceil\\
&\ + \left\lceil\frac{k}{m}\right\rceil(l_i+1) + b_{i+1}^{-k}\frac{r}{m}\bigg) + w\left(\frac{k}{m} + \epsilon_{i+1}b_{i+1}^{-k}\left\lceil\frac{|X_{i+1}|-r}{m}\right\rceil + b_{i+1}^{-k}\right) + \frac{z}{m} \Bigg). 
\end{align*}

\begin{lem}\labl{fourthlemma}
Let $B$ be a block of length $k$ in base $v_i$. If $k \le k_i$, then
$$
\left|\frac{N_{n,m,r}^Q(B,X)}{Q_{n,m,r}^{(k)}} - 1\right| <2g_i(w,z) + \frac{{S}_{n,m,r}^{(k)} - Q_{n,m,r}^{(k)}}{{Q}_{n,m,r}^{(k)}}.
$$
\end{lem}
\begin{proof}
Using the lower bound from \refeq{boundi} and \refeq{boundi2} of $N_{n,m,r}^{Q}(B,X)$, we get,
\begin{align*}
\left|\frac{N_{n,m,r}^Q(B,X)}{Q_{n,m,r}^{(k)}} - 1\right| &< 1 - \frac{(1-\epsilon_i)b_i^{-k}l_i\left\lceil\frac{|X_i|-r}{m}\right\rceil + (1-\epsilon_{i+1})b_{i+1}^{-k}\alpha\left\lceil\frac{|X_{i+1}|-r}{m}\right\rceil}{Q_{n,m,r}^{(j)}}\\
&< \frac{{S}_{n,m,r}^{(k)} -\left( (1-\epsilon_i)b_i^{-k}l_i\left\lceil\frac{|X_i|-r}{m}\right\rceil + (1-\epsilon_{i+1})b_{i+1}^{-k}\alpha\left\lceil\frac{|X_{i+1}|-r}{m}\right\rceil \right)}{Q_{n,m,r}^{(k)}} \frac{Q_{n,m,r}^{(k)}}{{S}_{n,m,r}^{(k)}} \left(\frac{{S}_{n,m,r}^{(k)}}{Q_{n,m,r}^{(k)}}\right)\\
&< \frac{{S}_{n,m,r}^{(k)} -\left( (1-\epsilon_i)b_i^{-k}l_i\left\lceil\frac{|X_i|-r}{m}\right\rceil + (1-\epsilon_{i+1})b_{i+1}^{-k}\alpha\left\lceil\frac{|X_{i+1}|-r}{m}\right\rceil \right)}{{S}_{n,m,r}^{(k)}}\cdot 2 = 2f_i(\alpha,\beta).
\end{align*}
Let
\begin{equation*}
\kappa = \left\lceil\frac{L_{i-1}-r}{m}\right\rceil + (1+\epsilon_{i})b_{i}^{-k}l_i\left\lceil\frac{|X_{i}|-r}{m}\right\rceil+ (1+\epsilon_{i+1})b_{i+1}^{-k}\alpha\left\lceil\frac{|X_{i+1}|-r}{m}\right\rceil + \left\lceil\frac{\beta}{m}\right\rceil + \left\lceil\frac{k\alpha}{m}\right\rceil + \left\lceil\frac{k}{m}\right\rceil(l_i+1).
\end{equation*}
Using the upper  bound from \refeq{boundn} of $N_{n,m,r}^{Q}(B,X)$,
\begin{align*}
\left|\frac{N_{n,m,r}^Q(B,X)}{Q_{n,m,r}^{(k)}} - 1\right| &\le \frac{\kappa}{Q_{n,m,r}^{(k)}} - 1 = \frac{(\kappa - {S}_{n,m,r}^{(k)}) + ({S}_{n,m,r}^{(k)} - Q_{n,m,r}^{(k)})}{Q_{n,m,r}^{(k)}}\\
&< \frac{\kappa - {S}_{n,m,r}^{(k)}}{Q_{n,m,r}^{(k)}}\frac{Q_{n,m,r}^{(k)}}{{S}_{n,m,r}^{(k)}} \left(\frac{{S}_{n,m,r}^{(k)}}{Q_{n,m,r}^{(k)}}\right) + \frac{{S}_{n,m,r}^{(k)} - Q_{n,m,r}^{(k)}}{Q_{n,m,r}^{(k)}} < 2\frac{\kappa - {S}_{n,m,r}^{(k)}}{{S}_{n,m,r}^{(k)}} + \frac{{S}_{n,m,r}^{(k)} - Q_{n,m,r}^{(k)}}{Q_{n,m,r}^{(k)}}.
\end{align*}
Next, $\kappa - {S}_{n,m,r}^{(k)}$ is equal to
\begin{align*}
&\ \left\lceil\frac{L_{i-1}-r}{m}\right\rceil + (1+\epsilon_{i})b_{i}^{-k}l_i\left\lceil\frac{|X_{i}|-r}{m}\right\rceil+ (1+\epsilon_{i+1})b_{i+1}^{-k}\alpha\left\lceil\frac{|X_{i+1}|-r}{m}\right\rceil + \left\lceil\frac{\beta}{m}\right\rceil\\
&\ + \left\lceil\frac{k\alpha}{m}\right\rceil + \left\lceil\frac{k}{m}\right\rceil(l_i+1) - \left(\sum_{j=1}^{i} b_j^{-k} \left\lceil\frac{l_j|X_j|-r}{m}\right\rceil + b_{i+1}^{-k}\left\lceil\frac{|u|-r}{m}\right\rceil\right)\\
&= \left(\left\lceil\frac{L_{i-1}-r}{m}\right\rceil - \sum_{j=1}^{i} b_j^{-k} \left\lceil\frac{l_j|X_j|-r}{m}\right\rceil\right)+ (1+\epsilon_{i})b_{i}^{-k}l_i\left\lceil\frac{|X_{i}|-r}{m}\right\rceil+ \left\lceil\frac{k\alpha}{m}\right\rceil + \left\lceil\frac{k}{m}\right\rceil(l_i+1) \\
&\ + (1+\epsilon_{i+1})b_{i+1}^{-k}\alpha\left\lceil\frac{|X_{i+1}|-r}{m}\right\rceil + \left\lceil\frac{\beta}{m}\right\rceil  - b_{i+1}^{-k}\left\lceil\frac{|u|-r}{m}\right\rceil\\
&< \left(\sum_{j=1}^{i-1}\frac{l_j|X_j|-r}{m} - \sum_{j=1}^{i-1} b_j^{-k} \frac{l_j|X_j|-r}{m} + 1\right)+ \left(  b_{i}^{-k}l_i\left\lceil\frac{|X_{i}|-r}{m}\right\rceil - b_i^{-k} \left\lceil\frac{l_j|X_j|-r}{m}\right\rceil   \right)\\
&\ + \epsilon_{i}b_{i}^{-k}l_i\left\lceil\frac{|X_{i}|-r}{m}\right\rceil+ \left\lceil\frac{k\alpha}{m}\right\rceil + \left\lceil\frac{k}{m}\right\rceil(l_i+1) + \epsilon_{i+1}b_{i+1}^{-k}\alpha\left\lceil\frac{|X_{i+1}|-r}{m}\right\rceil\\
&\ + \left(  b_{i+1}^{-k}\alpha\frac{|X_{i+1}|-r}{m} + b_{i+1}^{-k}\alpha + \frac{\beta}{m} + 1  - b_{i+1}^{-k}\frac{\alpha|X_{i+1}| + \beta -r}{m} \right)\\
&< \frac{L_{i-1}-r}{m} + 1+ \left(  b_{i}^{-k}\frac{l_i|X_{i}|-l_ir}{m} - b_i^{-k} \frac{l_j|X_j|-r}{m} + b_{i}^{-k}l_i \right) + \epsilon_{i}b_{i}^{-k}l_i\left\lceil\frac{|X_{i}|-r}{m}\right\rceil + \left\lceil\frac{k\alpha}{m}\right\rceil\\
&\ + \left\lceil\frac{k}{m}\right\rceil(l_i+1) + \epsilon_{i+1}b_{i+1}^{-k}\alpha\left\lceil\frac{|X_{i+1}|-r}{m}\right\rceil + \left(\frac{\beta}{m} + 1 + b_{i+1}^{-k}\left(\frac{r}{m} + \alpha \right)  \right)\\
&< \frac{L_{i-1}-r}{m} + 1+ b_{i}^{-k}l_i + \epsilon_{i}b_{i}^{-k}l_i\left\lceil\frac{|X_{i}|-r}{m}\right\rceil + \frac{k\alpha}{m} + 1 + \left\lceil\frac{k}{m}\right\rceil(l_i+1)\\
&+ \epsilon_{i+1}b_{i+1}^{-k}\alpha\left\lceil\frac{|X_{i+1}|-r}{m}\right\rceil + \frac{\beta}{m} + 1 + b_{i+1}^{-k} \alpha +  b_{i+1}^{-k}\frac{r}{m}\\
&= \left(\frac{L_{i-1}-r}{m} + 3+ b_{i}^{-k}l_i + \epsilon_{i}b_{i}^{-k}l_i\left\lceil\frac{|X_{i}|-r}{m}\right\rceil + \left\lceil\frac{k}{m}\right\rceil(l_i+1) + b_{i+1}^{-k}\frac{r}{m}\right)
 \\
& +\alpha\left(\frac{k}{m} + \epsilon_{i+1}b_{i+1}^{-k}\left\lceil\frac{|X_{i+1}|-r}{m}\right\rceil + b_{i+1}^{-k}\right) + \frac{\beta}{m}  = g_i(\alpha,\beta){S}_{n,m,r}^{(k)}.
\end{align*}
Note that $g_i(w,z) > f_i(w,z)$ since $\frac{L_{i-1}-r}{m} > {S}_{n,m,r}^{(k)}$ and the coefficients of $\alpha$ and $\beta$ in the numerator of $g_i(w,z)$ are greater than those in the numerator of $f_i(w,z)$. Thus,
\begin{equation*}
\left|\frac{N_{n,m,r}^Q(B,X)}{Q_{n,m,r}^{(k)}} - 1\right| <2g_i(w,z) + \frac{{S}_{n,m,r}^{(k)} - Q_{n,m,r}^{(k)}}{{Q}_{n,m,r}^{(k)}}.
\end{equation*}
\end{proof}

\begin{lem}\labl{fifthlemma}
If  $|X_{i+1}|> \frac{k+ 2m}{1-\epsilon_{i+1}b_{i+1}^{-k}}, |X_i| > \frac{9m + 2k}{1 - \frac{\epsilon_{i}b_{i}^{-k}}{m}}, |X_{i+1}| > \frac{kb_{i+1}^{k} + 3m}{1-\epsilon_{i+1}}$, and $|X_{i+1}|> \frac{\left(  k + 2m \right) b_{i+1}^k}{\epsilon_i -\epsilon_{i+1}} + r$, then $g_i(w,z) < g_i(0,|X_{i+1}|)$.
\end{lem}

Let $\epsilon^\prime_i = g(0, |X_{i+1}|)$. Then
$$
\left|\frac{N_{n,m,r}^Q(B,X)}{Q_{n,m,r}^{(k)}} - 1\right| <2\epsilon^\prime_i + \frac{{S}_{n,m,r}^{(k)} - Q_{n,m,r}^{(k)}}{{Q}_{n,m,r}^{(k)}}.
$$
Since we have shown $\frac{{S}_{n,m,r}^{(k)} - Q_{n,m,r}^{(k)}}{{Q}_{n,m,r}^{(k)}} \to 0$, we now need to show that $\epsilon^\prime_i \to 0$.

\begin{lem}\labl{sixthlemma}
$\lim_{i \to \infty}\epsilon^\prime_i = 0$.
\end{lem}
\begin{proof}
This can be shown by applying all of our estimates as well as \refeq{good2} and \refeq{good3} to
$$
\epsilon^\prime_i =  \frac{ \frac{L_{i-2}-r}{m} + \frac{l_{i-1}|X_{i-1}|}{m}   + 3+ b_{i}^{-k}l_i + \epsilon_{i}b_{i}^{-k}l_i\left\lceil\frac{|X_{i}|-r}{m}\right\rceil+ \left\lceil\frac{k}{m}\right\rceil(l_i+1) + b_{i+1}^{-k}\frac{r}{m}  + \frac{|X_{i+1}|}{m}}   {  {S}_{L_{i-1},m,r}^{(k)} + \frac{b_i^{-k}l_i|X_i|}{m} +|X_{i+1}| \left(\frac{b_{i+1}^{-k}}{m}\right)}.
$$
\end{proof}

\begin{proof}[Proof of \reft{APmain}]

For an arbitrary block $B$ in base $b$, we can estimate the bounds of $N_{n,m,r}^Q(B,X)$ by using \refl{firstlemma} and \refl{secondlemma}. The hypotheses of \refl{fifthlemma} are satisfied as we can pick a large enough $n$ s.t. $b_i$ and $\epsilon_i$ are small and by \refeq{good1}, the last condition is satisfied. We also assume that $n$ is large such that $k \le k_i$, $m \le m_i$, $b \le v_i$, and $S_{n,m,r}^{(k)}/Q_{n,m,r}^{(k)} < 2$  for all $\mrmr$. Using \refl{fourthlemma} and \refl{fifthlemma},
$$
\left|\frac{N_{n,m,r}^Q(B,X)}{Q_{n,m,r}^{(k)}} - 1\right| <2\epsilon^\prime_i + \frac{{S}_{n,m,r}^{(k)} - Q_{n,m,r}^{(k)}}{{Q}_{n,m,r}^{(k)}}.
$$
By \refl{thirdlemma} and \refl{sixthlemma},  $\epsilon_i^\prime \to 0$ and $\frac{{S}_{n,m,r}^{(k)} - Q_{n,m,r}^{(k)}}{{Q}_{n,m,r}^{(k)}}$ $\to 0$ as $i \to \infty$. We also know that $i \to \infty$ as $n \to \infty$, so
$$
\lim_{n \to \infty}\left|\frac{N_{n,m,r}^Q(B,X)}{Q_{n,m,r}^{(k)}} - 1\right| = 0.
$$
Therefore,
$$
\lim_{n \to \infty}\left|\frac{N_{n,m,r}^Q(B,X)}{Q_{n,m,r}^{(k)}}\right| = 1.
$$


\end{proof}
\subsection{Analysis of $C_{b,w}$}\labs{cbw}
In this section, we perform a careful analysis of the blocks $C_{b,w}$ that will allow us to prove \reft{construction}.  
We recall that $C_{b,w}$ is the concatenation of all blocks in base $b$ of length $w$ in lexicographic order, so $|C_{b,w}| = wb^w$.

\begin{lem} \labl{eigthlemma}
Let $b \geq 2$, $w \geq 1$.  Suppose that $M$ is an integer, $M! | w$,  $\mrmr$, $m \leq M$, and $n = |C_{b,w}|$. If $B$ is a block of length $k$, then
\begin{align}\labeq{eigthlemmaone}
\left\lfloor \frac{w-k+1}{m} \right\rfloor b^{w-k} &\le N_{n,m,r}(B,C_{b,w}) < \left(\frac{w}{m} + 2\right)b^{w-k};\\
\labeq{eigthlemmatwo}
\max\left(\left(\left\lfloor \frac{w-r}{m} \right\rfloor - k + 1\right)b^{w-k}, 0\right) &\le N_{n,m,r}^\prime(B,C_{b,w}) < \left\lceil\frac{w-r}{m}\right\rceil b^{w-k}.
\end{align}
\end{lem}
\begin{proof}
We first prove that \refeq{eigthlemmaone} holds.
To calculate the lower bound, the occurences of $B$ that only occur on the inside on the blocks of $C_{b,w}$ will be counted. For convenience, denote $C_{b,w}$ as the concatentation of $C_1C_2...C_{b^w}$ where each $C_i$ is the $i$th block concatenated in $C_{b,w}$.

To arrive at the lower bound of \refeq{eigthlemmaone}, we count the possible number occurences of $B$ in $C_1,\cdots,C_{b^w}$. First, assume that $B$ occurs in some $C_i$. In this $C_i$, there are $w-k+1$ positions $B$ could possibly start at in such a way that $B$ remains inside $C_i$. For example, if $B$ starts at the $(w-k+2) th$ position of $C_i$, then the last number of $B$ will be in $C_{i+1}$. Let $j$ be the position that $B$ begins on in $C_i$, thus $1 \le j \le w-k+1$.

Next for this block $B$ in $C_i$ that begins at position $j$, there are $b^{w-k}$ other blocks of $C_{b,w}$ such that $B$ begins on position $j$ because there are $w-k$ positions of $C_i$ that could contain any value less than $b$ yet still contain $B$ at starting position $j$. Thus there are a total of $(w-k+1)b^{w-k}$ positions any block $B$ could begin.

We must count occurences of $B$ at positions that are $r \pmod m$, so we would expect that approximately $\frac 1 m$ of these blocks will be counted. Each position that is counted in the previous $C_i$ is also counted in $C_{i+1}$ because $M!|w \implies m|w$ implies that for every $m$, each position that is counted in some $C_i$ will be counted in $C_{i+1}$ as well. To arrive at a lower bound, this number must be rounded down 
because some of these positions may straddle the boundary between $C_i$ and $C_{i+1}$.
Thus,
\begin{equation*}
\left\lfloor \frac{w-k+1}{m} \right\rfloor b^{w-k} \le \left\lfloor \frac{w-k+1}{m}b^{w-k} \right\rfloor\le N_{n,m,r}(B,X).
\end{equation*}

For the upper bound, we take the same value for the lower bound, but instead round this up for those special cases where there is an additional block $B$ found on positions $r \pmod m$. We then must add on the number of times $B$ will occur on the boundaries of these blocks.

To count occurences of $B$ on these boundaries, we must look at an occurance of $B$ between $C_i$ and $C_{i+1}$. Let $B$ occur on the position $j$ of $C_i$, and if $j+k > w$, then the block $B$ will occur on the boundary between $C_i$ and $C_{i+1}$. There will be $k-1$ values that $B$ could begin on in such a $C_i$. There are also $b^{w-k}$ other blocks in $C_{b,w}$ such that $B$ begins on the position $j$, because there are $w-k$ positions of $C_i$ that could hold any value less than $b$.

Since $C_{b,w}=C_1C_2\cdots C_{b^w}$ is written in lexicographic order, there is a direct relationship between the value of each position in $C_i$ and the value of each position in $C_{i+1}$. Given $B$,  we can determine the beginning $k-(w-j+1)$ positions of $C_i$ since there are at the beginning of $C_{i+1}$, and the ending $w-j+1$ positions of $C_{i+1}$ since they are at the end of $C_i$. Thus, out of the $w$ positions, we have $w - (w-j+1) - (k-(w-j+1)) = w-k$ that are truly undetermined. Thus, there are a total of $(k-1)b^{w-k}$ positions that $B$ could begin at on the boundary.

Only approximately $\frac{1}{m}$ of these will be counted though as $N_{n,m,r}(B,C_{b,w})$ only checks blocks that begin on positions $r \pmod m$. This value is then rounded up to obtain a upper bound for all cases.  Thus,

\begin{equation*}
N_{n,m,r}(B,X) \le \left\lceil \frac{w-k+1}{m} \right\rceil b^{w-k} + \left\lceil \frac{k-1}{m}  b^{w-k}\right\rceil < \left(\frac{w}{m} + 2\right)b^{w-k}.
\end{equation*}

We now prove \refeq{eigthlemmatwo}.
To calculate the lower bound, we will count the number of occurences of $B$ only inside blocks of $C_{b,w}$. Let $C_i$ be some block $B$ that occurs on positions given by an arithmetic progression. For example, the first value of $B$ occurs on some position $j$ and the second value of $B$ occurs on the postion $j+m$, and so on. Then let the block $D_i$ be constructed from the values inside of $C_i$ on positions congruent to $r \pmod m$.

The block $D_i$ will have length greater than or equal to $\lfloor \frac{w-r}{m}\rfloor$. So there are $\lfloor \frac{w-r}{m}\rfloor - k + 1$  possible positions the block $B$ could begin at in copies of $D_i$. If $\lfloor \frac{w-r}{m}\rfloor - k + 1<0$, we take $0$ as our lower bound for $N_{n,m,r}^\prime(B,C_{b,w})$.
Otherwise, similarily to our proof of \refeq{eigthlemmaone}, there will be $b^{w-k}$ other blocks of $C_{b,w}$ such that $B$ begins on position $j$ in $C_i$ and occurs in $D_i$. This is because there could be $w-k$ positions that could hold any value no larger than $b$ in such a block $C_i$. 

Additionally since $M!|w \implies m|w$, it is ensured that for every $m$, every position in $C_i$ that will be counted will also be counted in $C_{i+1}$. This ensures that no position that is counted in the previous block will be skipped, as this will lead to $\frac {1}{m}$ of the counted occurences of $B$. 
Thus we form the lower bound of 
\begin{equation*}
\max\left(\left(\left\lfloor \frac{w-r}{m} \right\rfloor - k + 1\right)b^{w-k}, 0\right) \le N_{n,m,r}^\prime(B,X).
\end{equation*}
The upper bound is found similarily to the other methods of counting blocks on boundaries like in \refl{eigthlemma}. We assume on every boundary of $D_i$ that $B$ occurs. Thus there will be a maximum of $\lceil \frac {k-1}{m} \rceil$ positions on $D_i$ where $B$ could begin on such that $B$ will occur on the boundary of $D_i$ and $D_{i+1}$. There will also be $b^{w-k}$ other blocks of $C_{b,w}$ such that $B$ begins on the position $j$ in $C_i$ and occurs between $D_i$ and $D_{i+1}$. Thus in total, there will be $\lceil \frac{k-1}{m} \rceil b^{w-k}$ occurences of $B$ on the boundaries. When this is added to the lower bound,
\begin{equation*}
N_{n,m,r}^\prime(B,X) \le \max\left(\left(\left\lfloor \frac{w-r}{m} \right\rfloor - k + 1\right)b^{w-k}, 0\right) + \left\lceil \frac{k-1}{m} \right\rceil b^{w-k} \le \left\lceil\frac{w-r}{m} \right\rceil b^{w-k}.
\end{equation*}

\end{proof}

\begin{lem}\labl{cbwI}
If $K<w$, $\epsilon \ge \frac{m + \max (k,m)}{w}$, and $M!|w$, then $C_{b,w}$ is $(\epsilon,K,M,\lambda_b)$-normal of type I.
\end{lem}
\begin{proof}

Let $n = |C_{b,w}|$. Then for $k \leq K$, $m \leq M$, and every block $B$ of length $k$ in base $b$
\begin{equation*}
\left\lfloor \frac{w-k+1}{m} \right\rfloor b^{w-k} > b^{-k}n \left(\frac{\frac{w-k+1}{m} -1}{w}   \right) = \lambda_b(B)\frac{n}{m} \left(1- \left(\frac{k}{w} - \frac{1}{w} + \frac{m}{w}\right)\right) > \lambda_b(B)\frac{n}{m} \left(1- \left(\frac{k+m}{w}\right)\right).
\end{equation*}
Since $M!|w \implies m|w$ and $n = wb^w$, we know that for an integer $c = \frac n m$ and $0 \le r < m$, then for some integer $d$, where $0 \le d < m$, $\left\lceil \frac {n-r}{m}\right\rceil = \left\lceil c + \frac{d}{m}\right\rceil = c = \frac{n}{m}$, or $\left\lceil \frac {n-r}{m}\right\rceil = \frac{n}{m}$. Thus,

\begin{equation*}
 \lambda_b(B)\left\lceil\frac{n-r}{m}\right\rceil \left(1- \frac{k+m}{w}\right) \ge \lambda_b(B)\left\lceil\frac{n-r}{m}\right\rceil (1-\epsilon).
\end{equation*}
For the upper bound, 
\begin{equation*}
\left(\frac{w}{m} + 2\right)b^{w-k} = b^{-k}n \left(\frac{\frac{w}{m} + 2}{w}   \right) = \lambda_b(B)\frac{n}{m} \left(1+ \frac{2m}{w}\right).
\end{equation*}
So,
\begin{equation*}
\lambda_b(B)\frac{n}{m} \left(1+ \frac{2m}{w}\right) = \lambda_b(B)\left\lceil\frac{n-r}{m}\right\rceil \left(1+ \frac{2m}{w}\right) \le \lambda_b(B)\left\lceil\frac{n-r}{m}\right\rceil (1+\epsilon)
\end{equation*}
and we have proved that $C_{b,w}$ is $(\epsilon,K,M,\mu)$-normal of type I.
\end{proof}

\begin{lem}\labl{cbwII}
If $K<w$, $b^{w} > \frac{r}{m} + \frac{3r}{w}$, $\epsilon \ge \frac{(k+1)m}{w}$, and $M!|w$, then $C_{b,w}$ is 
$(\epsilon,K,M,\lambda_b(B))$-normal of type II.
\end{lem}
\begin{proof}

Let $n=|C_{b,w}|$. For $k \leq K$, $m \leq M$, and blocks $B$ of length $k$ in base $b$
\begin{align*}
\left(\left\lfloor \frac{w-r}{m} \right\rfloor - k + 1\right)b^{w-k} &> b^{-k}n \left(\frac{\left\lfloor \frac{w-r}{m} 
\right\rfloor - k + 1}{w}   \right) > \lambda_b(B) n \left(\frac{ \frac{w-r}{m}  - k}{w}   \right)\\
&> \lambda_b(B) \frac{n}{m} \left(\frac{ w-r  - km}{w}   \right) = \lambda_b(B) \frac{n}{m} \left(1 - \frac{r}{w} - \frac{km}{w}  
 \right).
\end{align*}
Since $M!|w \implies m|w$ and $n = wb^w$, we know that for an integer $c = \frac n m$ and $0 \le r < m$, then for some integer $d$, where $0 \le d < m$, $\left\lceil \frac {n-r}{m}\right\rceil = \left\lceil c + \frac{d}{m}\right\rceil = c = \frac{n}{m}$, or $\left\lceil \frac {n-r}{m}\right\rceil = \frac{n}{m}$. Thus,

\begin{equation*}
\lambda_b(B) \frac{n}{m} \left(1 - \frac{r}{w} - \frac{km}{w}  
 \right) = \lambda_b(B)\left\lceil\frac{n-r}{m}\right\rceil \left(1 - \left(\frac{r+ km}{w}\right)   \right) \ge \lambda_b(B)\left\lceil
\frac{n-r}{m}\right\rceil (1-\epsilon).
\end{equation*}
For the upper bound, 
\begin{equation*}
\left\lceil\frac{w-r}{m}\right\rceil b^{w-k} \le b^{-k}n \left(\frac{w-r + m}{mw}   \right) = \lambda_b(B)\frac{n}{m} \left(1+ 
\frac{m}{w}\right).
\end{equation*}
Therefore,
\begin{equation*}
\lambda_b(B)\frac{n}{m} \left(1+ \frac{m}{w}\right) = \lambda_b(B)\left\lceil\frac{n-r}{m}\right\rceil \left(1+ \frac{m}
{w}\right) \le \lambda_b(B)\left\lceil\frac{n-r}{m}\right\rceil (1+\epsilon)
\end{equation*}
and we have proved that $C_{b,w}$ is $(\epsilon,K,M,\mu)$-normal of type II.
\end{proof}

We may now prove \reft{construction}.

\begin{proof}[Proof of \reft{construction}]
Since $X_i = C_{i,i!}$, $X_i$ is $(\epsilon_i,k_i,m_i,\mu_i)$-normal of types I and II by \refl{cbwI} and \refl{cbwII}. It is easy to show that \refeq{good1}, \refeq{good2}, and \refeq{good3} are satisfied noting that $|X_i|=i! \cdot i^{i!}$.
The conclusion follows from \reft{APmain}.

\end{proof}

\subsection{Proof of \reft{normalorders} and \reft{APnormalorders}}

We will first need the following two elementary lemmas.

\begin{lem}\labl{tcorr}
Let $L$ be a real number and $(a_n)_{n=1}^\infty$ and $(b_n)_{n=1}^\infty$ be two sequences of positive real numbers such that
$$
\sum_{n=1}^{\infty} b_n=\infty \hbox{ and } \lim_{n \to \infty} \frac {a_n} {b_n}=L.
$$
Then
$$
\lim_{n \to \infty} \frac {a_1+a_2+\ldots+a_n} {b_1+b_2+\ldots+b_n}=L.
$$
\end{lem}

\begin{lem}\labl{tcorr2}
Let $L$ be a real number and $(a_n)_{n=1}^\infty$ and $(b_n)_{n=1}^\infty$ be two sequences of positive integers.  Let $(c_t)_{t=1}^\infty$ be an increasing sequence of positive integers.  Set $A_t=\sum_{n=c_t}^{c_{t+1}-1} a_n$ and $B_t=\sum_{n=c_t}^{c_{t+1}-1} b_n$.  If $\lim_{t \to \infty} \frac {A_1+A_2+\cdots+A_t} {B_1+B_2+\cdots+B_t}=L,\  \sum_{t=1}^\infty A_t=\sum_{t=1}^\infty B_t=\infty,$ and $\lim_{t \to \infty} \frac {A_{t+1}} {A_1+A_2+\cdots+A_t}=\lim_{t \to \infty} \frac {B_{t+1}} {B_1+B_2+\cdots+B_t}=0$, then
$$
\lim_{n \to \infty} \frac {a_1+a_2+\ldots+a_n} {b_1+b_2+\ldots+b_n}=L.
$$
\end{lem}

We may now proceed with the proof of \reft{normalorders}.

\begin{proof}[Proof of \reft{normalorders}]
Since $\al_j | b_i$, $Q$ is a basic sequence.
Let $L_i$ be definitioned as in \reft{thm3.1} and let $i=i(n)$ be the integer that satisfies $L_{i-1} < n \leq L_i$. Let $k \in \mathbb{N}$ and set 
$$
P_{j,k}=\sum_{v=0}^{d-1} \frac {1} {p_{jd+v+1}p_{jd+v+2}\cdots p_{jd+v+k}} \hbox{ and }Q_{j,k}=\sum_{v=0}^{d-1} \frac {1} {q_{jd+v+1}q_{jd+v+2}\cdots q_{jd+v+k}}.
$$
Note that   for large enough $j$, we have 
$p_n \in \{b_{i(jd+1)},b_{i(jd+1)+1}\}$ for $jd+1 \leq n \leq (j+1)d+k-1$.  Thus, since $b_i$ is non-decreasing
$\frac {d} {b_{i+1}^k} \leq P_{j,k} \leq \frac {d} {b_i^k}$.
Similarly,
$$
\frac {\sum_{j=0}^{t-k} c_jc_{j+1}\cdots c_{j+k-1}} {b_i^k} < Q_{j,k} < \frac {\sum_{j=0}^{t-k} c_jc_{j+1}\cdots c_{j+k-1}} {b_i^k}+\frac {d-(t-k+1)} {2^{jd+1} b_i^{k-1}}\cdot \max_{0 \leq r \leq t-1} c_r^{k-1},
$$
so
\begin{align}\labeq{normalorder1}
\frac {Q_{j,k}} {P_{j,k}} 
&< \frac {\frac {\sum_{j=0}^{t-k} c_jc_{j+1}\cdots c_{j+k-1}} {b_i^k}+\frac {d-(t-k+1)} {2^{jd+1} b_i^{k-1}}\cdot \max_{0 \leq r \leq t-1} c_r^{k-1}} {\frac {d} {b_{i+1}^k}}\\
&=\left(\frac {b_{i+1}} {b_i}\right)^k \cdot \frac {\sum_{j=0}^{t-k} c_jc_{j+1}\cdots c_{j+k-1}} {d}+\frac {d-(t-k+1)} {d} \cdot \left(\frac {b_{i+1}} {b_i} \right)^{k-1} \cdot \frac {b_{i+1}} {2^{jd+1}}\cdot \max_{0 \leq r \leq t-1} c_r^{k-1}
\end{align}
and
\begin{equation}\labeq{normalorder2}
\frac {Q_{j,k}} {P_{j,k}}> \frac {\frac {\sum_{j=0}^{t-k} c_jc_{j+1}\cdots c_{j+k-1}} {b_i^k}} {\frac {d} {b_i^k}}
=\frac {\sum_{j=0}^{t-k} c_jc_{j+1}\cdots c_{j+k-1}} {d}.
\end{equation}
However, $\frac {b_{i+1}} {b_i} \to 1$ and \refeq{good1} holds, so $\lim_{j \to \infty} \frac {b_{i(jd+1)+1}} {2^{jd+1}} \to 0$.  Therefore,
\begin{equation}\labeq{mainpsithing2}
\lim_{j \to \infty} \frac {Q_{j,k}} {P_{j,k}}=\frac {\sum_{j=0}^{t-k} c_jc_{j+1}\cdots c_{j+k-1}} {d},
\end{equation}
by \refeq{normalorder1} and \refeq{normalorder2}.  
Using \refl{tcorr}, \refl{tcorr2}, \refeq{mainpsithing2}, the fact that $\lim_{j \to \infty} \frac {Q_{j,k}} {P_{j,k}}$ exists, and that $P$ and $Q$ are $k$-divergent, we have that
\begin{equation}\labeq{limssame}
\lim_{n \to \infty} \frac {\qnk} {\pnk}= \lim_{j \to \infty} \frac {Q_{j,k}} {P_{j,k}}=\frac {\sum_{j=0}^{t-k} c_jc_{j+1}\cdots c_{j+k-1}} {d}.
\end{equation}
Let $k \in \{1,\cdots,t\}$ and let $B$ be any block of length $k$.  By \reft{mainpsi} 
\begin{equation}\labeq{mainpsithing}
N_n^Q(B,\ppqewx)=N_n^P(B,\eta(W,X))+O(1).
\end{equation}
Since $\eta(W,X) \in \Nk{Q}{k}$ by \reft{thm3.1}, we know that $\lim_{n \to \infty} \frac {N_n^P(B,\eta(W,X))} {\pnk}=1$.  Thus, by \refeq{limssame} and \refeq{mainpsithing}
$$
\lim_{n \to \infty} \frac {N_n^Q(B,\ppqewx)} {\qnk}=\lim_{n \to\infty} \frac {N_n^P(B,\eta(W,X))+O(1)} {\pnk} \cdot \frac {\pnk} {\qnk}=\frac {d} {\sum_{j=0}^{t-k} c_jc_{j+1}\cdots c_{j+k-1}},
$$
which is equal to $1$ if $k \in A$ and not equal to $1$ if $k \in B$.  Thus, $\ppqewx \in \Nk{Q}{k}$ if and only if $k \in A$.
\end{proof}

\begin{proof}[Proof of \reft{APnormalorders}]
The proof follows along similar, but slightly more complex lines than the proof of \reft{normalorders}, so we omit the details.  It should be noted that \reft{mainpsi} is used to prove \refeq{APIorders1} and \refeq{APIorders2} and \refc{mainpsi2} is used to prove \refeq{APIIorders1} and \refeq{APIIorders2}.
\end{proof}

\subsection{Proof of \reft{ap2notnormal}}\labs{ap2notnormal}

We note that with $m=k, t=2k^2, d=2k^2(k+1), c_0=c_1=\cdots=c_{k-1}=c_{k^2}=c_{k^2+1}=\cdots=c_{k^2+k-1}=2k$, and $c_n=1$ for all other $n$, we have $ c_{r+jm}c_{r+(j+1)m}\cdots c_{r+(j+k-1)m}=2k$ for all $j, m$, and $r$.  Also,
$$
\floor{\frac {t-r-1} {m}}=\floor{\frac {2k^2-r-1} {k}}=\floor{\frac {2k^2-(k-1)-1} {k}}=2k-1 \hbox{ for all }r \in [0,k-1].
$$
Thus, there are $\left(\floor{\frac {t-r-1} {m}}-k+1\right)-0+1=2k-1-k+1+1=k+1$ terms in the sum in \refeq{APIIorders1}, each of which is equal to $2k$.  Therefore,
$$
\sum_{j=0}^{\floor{\frac {t-r-1} {m}}-k+1} c_{r+jm}c_{r+(j+1)m}\cdots c_{r+(j+k-1)m}=(k+1) \cdot 2k=\frac {2k^2(k+1)}{k}=d/m,
$$
so \refeq{APIIorders1} holds for all $r \in [0,k-1]$ and $\psi_{R_t,Q}(\zeta_t) \in \bigcap_{r=0}^{k-1} \Nkt{Q}{k,k,r}{2}=\Nkt{Q}{k}{2}$.

To show that $\psi_{R_t,Q}(\zeta_t) \notin \Nk{Q}{k}$, we first consider the case where $k=2$.  Here,
\begin{align*}
\sum_{j=0}^{t-k} c_jc_{j+1}\cdots c_{j+k-1}=\sum_{j=0}^6 c_j c_{j+1}&=c_0c_1+c_1c_2+c_2c_3+c_3c_4+c_4c_5+c_5c_6+c_6c_7\\
&=16+4+1+4+16+4+1=46 \neq 24=2 \cdot 2^2 (2+1)=d.
\end{align*}
It is easy to show that $(2k)^k>2k^2(k+1)$ for all $k \geq 3$.
Hence, if $k \geq 3$, then
$$
\sum_{j=0}^{t-k} c_jc_{j+1}\cdots c_{j+k-1}>c_0c_1\cdots c_{k-1}=(2k)^k>2k^2(k+1)=d.
$$
In either case, \refeq{normalorders2} holds, so $\psi_{R_t,Q}(\zeta_t) \notin \Nk{Q}{k}$.

\section{Further questions}\labs{conjectures}
Several conjectures are presented in this section.  We will intentionally omit many details in this section as the computations supporting the likelihood of some of these assertions would make this section disproportionately long.

It is likely that there exist much stranger examples of numbers than those that can be constructed with \reft{normalorders} and \reft{APnormalorders}.  

\begin{conjecture}\labj{APabnormal}
There exists a basic sequence $Q$ that is infinite in limit and fully divergent of type I  and a real number $x$ is that $Q$-normal but ``AP $Q$-abnormal''.  That is,
$$
x \in \N{Q} \backslash \bigcup_{k=2}^\infty \bigcup_{r=0}^{k-1} \NAPIkmr{Q}{k}{k}{r}.
$$
\end{conjecture}
We may think of this number described in \refj{APabnormal} as being a normal number whose digits along {\it every} arithmetic progression do not behave in a way that one would expect from a normal number.  We will outline the evidence for \refj{APabnormal} and provide an even stronger conjecture.

\begin{definition}
\footnote{
Note that in base~$b$, where $q_n=b$ for all $n$,
 the corresponding notions of $Q$-normality, $Q$-ratio normality, and $Q$-distribution normality are equivalent. 
}
Let $Q \in \NN$ and put $ T_{Q,n}(x)=\left(\prod_{j=1}^n q_j\right) x \pmod{1}$.  A real number~$x$ is {\it $Q$-distribution normal} if
the sequence $(T_{Q,n}(x))_{n=0}^\infty$ is uniformly distributed mod $1$.  Let $\DNQ$ be the set of $Q$-distribution normal numbers.
\end{definition}

\begin{figure}
\begin{tikzpicture}[>=stealth',shorten >=1pt,node distance=3.8cm,on grid,initial/.style    ={}]
  \node[state]          (NQ)                        {$\mathsmaller{\NQ}$};
  \node[state]          (RNQ) [left =of NQ]    {$\mathsmaller{\RNQ}$};
  \node[state]          (NQRNQ) [above right =of NQ]    {$\mathsmaller{\NQ \cap \RNQ}$};
  \node[state]          (RNQDNQ) [above left=of RNQ]    {$\mathsmaller{\RNQ \cap \DNQ}$};
  \node[state]          (NQDNQ) [above left =of NQ]    {$\mathsmaller{\NQ \cap \DNQ}$};
  \node[state]          (NQRNQDNQ) [above right =of NQDNQ]    {$\mathsmaller{\NQ \cap \RNQ \cap \DNQ}$};
  \node[state]          (DNQ) [above right=of RNQDNQ]    {$\mathsmaller{\DNQ}$};
\tikzset{mystyle/.style={->,double=black}} 
\tikzset{every node/.style={fill=white}} 
\path (RNQDNQ)     edge [mystyle]    (RNQ)
      (RNQDNQ)     edge [mystyle]     (DNQ)
      (NQ)     edge [mystyle]     (RNQ)
      (NQDNQ)     edge [mystyle]     (RNQDNQ)
      (NQDNQ)     edge [mystyle]     (NQ);
\tikzset{mystyle/.style={<->,double=black}}
\path (NQRNQDNQ)     edge [mystyle]    (NQDNQ)
	(NQ)     edge [mystyle]    (NQRNQ);
\end{tikzpicture}
\caption{}
\labf{figure1}
\end{figure}

We refer to the directed graph in \reff{figure1} for the complete containment relationships between these notions when $Q$ is infinite in limit and fully divergent.  The vertices are labeled with all possible intersections of one, two, or three choices of the sets $\NQ$, $\RNQ$, and $\DNQ$.  The set labeled on vertex $A$ is a subset of the set labeled on vertex $B$ if and only if there is a directed path from vertex $A$ to vertex $B$.\footnote{The underlying undirected graph in \reff{figure1} has an isomorphic copy of complete bipartite graph $K_{3,3}$ as a subgraph.  Thus, it is not planar and the directed graph that connects two vertices if and only if there is a containment relation between the two labels is more difficult to read.}  For example, $\NQ \cap \DNQ \subset \RNQ$, so all numbers that are $Q$-normal and $Q$-distribution normal are also $Q$-ratio normal.   The relations described in \reff{figure1} are best derived through use of \reft{mainpsi}.
More information can be found in \cite{ppq1}. 

Let $\dimh{S}$ be the Hausdorff dimension of a set $S \subseteq \mathbb{R}$.  The following is a special case of a much more general theorem proven in \cite{AireyMance1}.
\begin{thrm}\labt{APDN}
Suppose that $Q$ is infinite in limit and put
$$
A=\left\{x \in \mathbb{R} : (T_{Q,mn+r}(x))_{n=1}^\infty \hbox{ is not uniformly distributed mod }1 \forall m \in \mathbb{N}_2, r \in [0,m-1]\right\}.
$$
Then $\dimh{\DN{Q} \cap A}=1$.
\end{thrm}
Thus, the following stronger version of \refj{APabnormal} seems likely to be true.
\begin{conjecture}\labj{APabnormalHD}
If $Q$ is infinite in limit and fully divergent of type I, then
$$
\dimh{\N{Q} \backslash \bigcup_{k=2}^\infty \bigcup_{r=0}^{k-1} \NAPIkmr{Q}{k}{k}{r}}=1.
$$
\end{conjecture}
While there is a substantial amount of machinery that may be used to prove statements like \reft{APDN}, neither of the authors is aware of anything that could be used to establish \refj{APabnormalHD}.  However,  \refj{APabnormal} may lend itself to more immediate analysis.  
\begin{conjecture}\labj{system}
\footnote{
Ryan Greene has verified \refj{system} by computer up to $t=100$ when $\vec{\epsilon}=(0,0,\cdots,0)$.  It is plausible that small pertubations of the right hand side of this system will still allow for solutions in this critical region.}
Let $t  \in \mathbb{N}_3$ and let $\vec{\epsilon}=(\epsilon_1,\cdots,\epsilon_{t}) \in \mathbb{R}^t$.  Then the nonlinear system of equations
\begin{align*}
c_0+c_1+c_2+c_3+\cdots+c_{t-1}&=2t+\epsilon_1\\
c_0c_1+c_1c_2+c_2c_3+c_3c_4+\cdots+c_{t-2}c_{t-1}&=2t+\epsilon_2\\
c_0c_1c_2+c_1c_2c_3+c_2c_3c_4+c_3c_4c_5+\cdots+c_{t-3}c_{t-2}c_{t-1}&=2t+\epsilon_3\\
\cdots\\
c_0c_1\cdots c_{t-1}&=2t+\epsilon_t
\end{align*}
has a solution $(c_0,c_1,\cdots,c_{t-1}) \in [t,t+1] \times \left[1+\frac {1} {2t},1+\frac {1} {t-1}\right]^{t-1}$ whenever $\|\vec{\epsilon}\|$ is sufficiently small.
\end{conjecture}
If \refj{system} is true, then we can prove \refj{APabnormal} by means of a more sophisticated version of \reft{normalorders} and \reft{APnormalorders} that allows for members of $\N{Q}$ to be constructed.  Furthermore, a proof of \refj{system} would prove the following conjecture.

\begin{conjecture}\labj{finiteorders}
Let $t \geq 2$ be an integer and suppose that $\{1,\cdots,t\}=A \cupdot B$.  There there exists a real number $x$ and basic sequence $Q$ that is $t$-divergent where
$$
x \in \bigcap_{k \in A} \Nk{Q}{k} \backslash \bigcap_{j \in B} \Nk{Q}{j}.
$$
\end{conjecture}

It is possible that careful analysis of the solutions of the system described in \refj{system} could settle the following much more surprising conjecture.

\begin{conjecture}\labj{infiniteorders}
Suppose that $\mathbb{N}=A \cupdot B$ is a partition of the set of natural numbers.  Then there exists a real number $x$ and basic sequence $Q$ that is fully divergent where
$$
x \in \bigcap_{k \in A} \Nk{Q}{k} \backslash \bigcap_{j \in B} \Nk{Q}{j}.
$$
\end{conjecture}
The authors feel that \refj{APabnormal} and \refj{finiteorders} are almost surely true, but hesitate to suggest that \refj{infiniteorders} is true.  We also wish to state the following extension of \refj{finiteorders} and \refj{infiniteorders}.

\begin{conjecture}\labj{HDorders}
If $Q$ is infinite in limit and $t$ divergent and $\{1,\cdots,t\}=A \cupdot B$, then
$$
\dimh{\bigcap_{k \in A} \Nk{Q}{k} \backslash \bigcap_{j \in B} \Nk{Q}{j}}.
$$ 
If $Q$ is infinite in limit and fully divergent and $\mathbb{N}=C \cupdot D$, then
$$
\dimh{\bigcap_{k \in C} \Nk{Q}{k} \backslash \bigcap_{j \in D} \Nk{Q}{j}}.
$$ 
\end{conjecture}
However, neither of the authors anticipate an approach that could be used to settle \refj{HDorders}.  

\section*{Acknowledgments}

Research of the second author is partially supported by the U.S. NSF grant DMS-0943870.  The authors are indebted to the referee for many valuable suggestions.

\appendix \section{ }

I. Niven and H. S. Zuckerman wrote in \cite{NivenZuckerman}:
\begin{quotation}
Let $R$ be a real number with fractional part $.x_1x_2x_3\cdots$ when written to scale $r$.  Let $N(b,n)$ denote the number of occurrences of the digit $b$ in the first $n$ places.  The number $R$ is said to be {\it simply normal} to scale $r$ if $\lim_{n \to \infty} \frac {N(b,n)} {n}=\frac {1} {r}$  for each of the $r$ possible values of $b$; $R$ is said to be {\it normal} to scale $r$ if all the numbers $R, rR, r^2R,\cdots$ are simply normal to all the scales $r, r^2, r^3, \cdots$.  These definitions, for $r=10$, we introduced by \'{E}mile Borel \cite{BorelNormal}, who stated (p. 261) that ``la propri\'{e}t\'{e} caract\'{e}ristique'' of a normal number is the following: that for any sequence $B$ whatsoever of $v$ specified digits, we have
\begin{equation}\labeq{quote}
\lim_{n \to \infty} \frac {N(B,n)} {n}=\frac {1} {r^v},
\end{equation}
where $N(B,n)$ stands for the number of occurrences of the sequence $B$ in the first $n$ decimal places
\ldots
If the number $R$ has the property \refeq{quote} then any sequence of digits
$
B=b_1b_2\cdots b_v
$
appears with the appropriate frequency, but will the frequencies all be the same for $i=1,2,\cdots,v$ if we count only those occurrences of $B$ such that $b_1$ is an $i,i+v,i+2v,\cdots-th$ digit?  It is the purpose of this note to show that this is so, and thus to prove the equivalence of property \refeq{quote} and the definition of normal number.
\end{quotation}
It is not difficult to see how the equivalent definition of normality introduced in \reft{basebnormalii} may be confused with the notion discussed in \reft{basebnormalii}.

\bibliographystyle{plain}


\end{document}